\documentclass[11pt, oneside]{amsart}   	% use "amsart" instead of "article" for AMSLaTeX format
\usepackage{amsmath,amsthm,amscd,amsfonts, amssymb, mathrsfs}

\usepackage{geometry}                		% See geometry.pdf to learn the layout options. There are lots.
\usepackage{graphicx}				% Use pdf, png, jpg, or eps? with pdflatex; use eps in DVI mode
\makeatletter
\@namedef{subjclassname@2020}{\textup{2020} Mathematics Subject Classification}
\makeatother

\newcommand{\sect}[1]{\section{#1}\setcounter{equation}{0}}
\newcommand{\subsect}[1]{\subsection{#1}}

\font\mbn=msbm10 scaled \magstep1
\font\mbs=msbm7 scaled \magstep1
\font\mbss=msbm5 scaled \magstep1
\newfam\mbff
\textfont\mbff=\mbn
\scriptfont\mbff=\mbs
\scriptscriptfont\mbff=\mbss   

\newcommand{\Di}      {\mathbb{D}}
\newcommand{\RR}       { \mathbb{R}}

\newcommand{\N}       { \mathbb{N}}
\newcommand{\Z}        {\mathbb{Z}  }  
\newcommand\Co           {{\mathbb C}}

\newtheorem{Th}{Theorem}[section]
\newtheorem{Lm}[Th]{Lemma}
\newtheorem{C}[Th]{Corollary}

\newtheorem{Prop}[Th]{Proposition}
\newtheorem{R}[Th]{Remark}

\newtheorem*{Lemma A}{Lemma A}
\newtheorem*{Lemma B}{Lemma B}
\newtheorem*{Lemma C}{Lemma C}

\newtheorem*{Th A}{Theorem A}
\newtheorem*{Th B}{Theorem B}
\begin{document}

\title[$L^\infty$ Estimates for the Banach-valued  
 $\bar\partial$-problem in a Disk]{$L^\infty$ Estimates for the Banach-valued  
 $\bar\partial$-problem in a Disk}
\author{Alexander Brudnyi}
\address{Department of Mathematics and Statistics\newline
\hspace*{1em} University of Calgary\newline
\hspace*{1em} Calgary, Alberta, Canada\newline
\hspace*{1em} T2N 1N4}
\email{abrudnyi@ucalgary.ca}


\thanks{Research is supported in part by NSERC}

\begin{abstract} 
We study the differential equation $\frac{\partial G}{\partial\bar z}=g$  with an unbounded Banach-valued Bochner measurable function $g$ on the open unit disk $\Di\subset\Co$. We prove that under some conditions on the growth and essential support of $g$ such equation has a bounded solution given by a continuous linear operator. The obtained results are applicable to the Banach-valued corona problem for the algebra of bounded holomorphic functions on $\Di$ with values in a complex commutative unital Banach algebra.
\end{abstract}

\date{}

\maketitle
\sect{Formulation of Main Results }
 Let $K$ be a Lebesgue measurable subset of the open unit disk $\Di\subset\Co$ and $X$ be a complex Banach space. Two $X$-valued functions on $\Di$ are equivalent if they coincide a.e. on $\Di$. The complex Banach space
$L^\infty(K,X)$ consists of equivalence classes of Bochner measurable essentially bounded functions $f:\Di\to X$ equal $0$ a.e.\,\,on $\Di\setminus K$ equipped with norm $\|f\|_\infty:={\rm ess}\sup_{z\in K}\|f(z)\|_X$.

In the present paper we study the differential equation
\begin{equation}\label{eq1.1}
\frac{\partial F}{\partial\bar z}=\frac{f(z)}{1-|z|^2},\qquad |z|<1,\quad f\in L^\infty(K,X).
\end{equation}
Such equations play an essential role in the area of the Banach-valued corona problem for the algebra $H^\infty(\Di,A)$ of bounded holomorphic functions on $\Di$ with values in a complex commutative unital Banach algebra $A$, see, e.g., \cite{Br2}, \cite{Br3} and references therein.
We prove that for a certain class of sets $K$  equation \eqref{eq1.1}
has a weak solution $F\in L^\infty(\Di, X)$, i.e., such that
for every $C^\infty$ function $\rho$ with compact support in $\Di$
\begin{equation}\label{eq1.2}
\iint \limits_{\Di}F(z)\cdot\frac{\partial \rho(z)}{\partial\bar{z}}\,dz\wedge d\bar{z}=-\iint\limits_{\Di}\frac{f(z)}{1-|z|^2}\,\cdot  \rho(z)\,dz\wedge d\bar{z},
\end{equation}
given by a bounded linear operator $L_K^X: L^\infty(K,X)\rightarrow  L^\infty(\Di,X)$. The operator $L_K^X$
is constructed explicitly that provides effective bounds of its norm in terms of some characteristics of $K$.

To formulate our results, recall that a sequence $\{z_n\}\subset\Di$ is said to be {\em interpolating} for $H^\infty$, the Banach space of bounded holomorphic functions on $\Di$, if every interpolation problem
\begin{equation}\label{eq1.3}
g(z_n)=a_n,\quad n\ge 1,
\end{equation}
with a bounded data $\{a_n\}\subset\Co$ has a solution $g\in H^\infty$.

By the Banach open mapping theorem, there is a constant $M$ such that problem \eqref{eq1.3} has a solution $g\in H^\infty$ satisfying
\[
\|g\|_{\infty}\le M\sup_{n}|a_n|.
\]
The smallest possible $M$ is said to be the constant of interpolation of $\{z_n\}$.

Clearly, every finite subset of $\Di$ is an interpolating sequence for $H^\infty$. In general, by the Carleson theorem, see, e.g., \cite[Ch.\,VII,\,Thm.\,1.1]{Ga}, a sequence $\zeta=\{z_n\}\subset\Di$ is interpolating for $H^\infty$ if and only if for some $\delta>0$ the characteristic 
\begin{equation}\label{eq1.4}
\delta(\zeta):=\inf_{k}\prod_{j,\,j\ne k}\rho(z_j,z_k)\ge\delta,\footnote{here $\delta(\zeta)=1$ if $\zeta$ consists of one element.}
\end{equation}
where
\[
\rho(z,w):=\left|\frac{z-w}{1-\bar w z}\right|,\qquad z,w\in\mathbb D,
\]
is the pseudohyperbolic metric on $\mathbb D$.

In turn, the constant of interpolation $M_\zeta$ of $\zeta$ satisfies
\begin{equation}\label{eq1.5}
\frac{1}{\delta}\le M_\zeta\le\min\left\{\frac{2e}{\delta}\log\frac{e}{\delta^2}\,,\left(\frac{1+\sqrt{1-\delta^2}}{\delta}\right)^{\!2}\right\}.
\end{equation}
(The first term in the braces is an upper bound of the P. Jones linear interpolation operator \cite[Thm.\,6]{J}  obtained in \cite{VGK} by elementary arguments,  and the second one is due to Earl \cite[Thm.\,2]{E}.)\smallskip

A subset $S$ of  a metric space $(\mathcal M,d)$
 is said to be $\epsilon$-{\em separated} if $d(x,y)\ge\epsilon$ for all $x, y\in S$, $x\ne y$. A maximal $\epsilon$-separated subset of $\mathcal M$ is said to be an $\epsilon$-{\em chain}. Thus, if $S\subset \mathcal M$ is an $\epsilon$-chain, then $S$ is $\epsilon$-separated and for every $z\in \mathcal M\setminus S$ there is 
 $x\in S$ such that $d(z,x)<\epsilon$.
 Existence of $\epsilon$-chains  follows from the Zorn lemma.\smallskip

 A subset $K\subset\Di$ is said to be {\em quasi-interpolating}, if an
 $\epsilon$-chain of $K$, $\epsilon\in (0,1)$,  with respect to $\rho$ is an interpolating sequence for $H^\infty$. (In fact, in this case every $\epsilon$-chain of $K$, $\epsilon\in (0,1)$, with respect to $\rho$ is an interpolating sequence for $H^\infty$, this easily follows from \cite[Ch.\,X,\,Cor.\,1.6, Ch.\,VII,\,Lm.\,5.3]{Ga}.)
 
 An important example of a quasi-interpolating set is a pseudohyperbolic neighbourhood (see \eqref{eq1.7} for its definition) of a Carleson contour used in the proof of the corona theorem \cite{C}, \cite{Z}.

 Given a complex Banach space $X$ and a subset $U\subset\Di$ we denote
 by $C_\rho(U,X)$ the Banach space of bounded continuous functions $f:U\rightarrow X$  uniformly continuous with respect $\rho$ equipped with norm $\|f\|_\infty:=\sup_{z\in U}\|f(z)\|_X$.
 
 We are ready to formulate the main result of the paper.
\begin{Th}\label{teo1.1}
Suppose a quasi-interpolating set $K\subset\Di$ is Lebesgue measurable and $\zeta=\{z_j\}$ is an $\epsilon$-chain of $K$, $\epsilon\in (0,1)$, with respect to $\rho$ such that $\delta(\zeta)\ge\delta>0$.\\
There is  a bounded linear operator $L_K^X: L^\infty(K,X)\to C_\rho(\Di,X)$ of norm 
\begin{equation}\label{eq1.6}
\|L_K^X\|\le \frac{c\epsilon}{1-\epsilon}\cdot\max\left\{1,\frac{\log\frac{1}{\delta}}{(1-\epsilon_*)^2}\right\},\quad \epsilon_*:=\max\left\{\frac 1 2,\epsilon\right\},
\end{equation}
for a numerical constant $c<5^2\cdot 10^6$
such that  for every $f\in L^\infty(K,X)$ the function $L_K^X f$ is a weak solution of equation \eqref{eq1.1}. \smallskip
 
The operator $L_K^X$ has the following properties:
\begin{itemize}
\item[(i)]
If $T: X\rightarrow Y$ is a bounded linear operator between complex Banach spaces, then 
\[
TL_K^X=L_K^Y T,
\]
where $(Tf)(z):=T(f(z))$, $z\in\Di$, $f:\Di\to X$;\smallskip
\item[(ii)]
If $f\in  L^\infty(K,X)$ has a compact essential range, then 
the range of $L_K^X f$ is relatively compact;\smallskip
\item[(iii)]
If $f\in L^\infty(X,K)$ is continuously differentiable on an open set $U\subset\Di$, then $L_K^X f$ is continuously differentiable on $U$. 
\end{itemize}
\end{Th}
\begin{R}\label{rem1.2}
{\rm (1) The construction of the operator $L_K^X$ involves integration of elements of $L^\infty(K,X)$. This implies that if
 ${\rm supp}\, K$ is the essential support of the characteristic function of $K$ and $K':=K\cap  {\rm supp}\, K$ (-- a Lebesgue measurable subset of full measure in $K$), then the correspondence $f\mapsto f|_{K'}$ determines an isometric isomorphism $I_{K'}^K: L^\infty(K,X)\to L^\infty(K',X)$ such that $L_K^X=L_{K'}^X\circ I_{K'}^K$.
Thus without loss of generality we may tacitly assume hereafter that $K\subset {\rm supp}\, K$.\smallskip

\noindent (2)
 A nonquantitative version of Theorem \ref{teo1.1} for a class of $C^\infty$ functions $f$ with  relatively compact images was presented earlier in \cite[Thm.\,3.5]{Br2}. }
 \end{R}
 
Let us formulate another property of the operator $L_K^X$.\smallskip 

In the sequel, we use the following notation.
 \[
 \Di_r(x):=\{z\in\Co\, :\, |z-x|<r\},\quad \Di_r:=\Di_r(0)\quad {\rm and}\quad D(x,r):=\{z\in\Co\, :\, \rho(z,x)<r\}.
 \]

For an interpolating sequence $\zeta\subset\Di$  by $B_\zeta$ we denote the interpolating Blaschke product having simple zeros at points of $\zeta$.

For a subset $S\subset\Di$ its open $\nu$-pseudohyperbolic neghbourhood, $\nu\in (0,1)$, is given by the formula
\begin{equation}\label{eq1.7}
[S]_\nu:=\left\{y\in\Di\, :\, \inf_{z\in S}\rho(y,z)<\nu\right\}.
\end{equation}

We denote by $\bar{S}$ the closure of $S$.

Finally, $\mathscr B(K,X)$ stands for the Banach space of bounded linear operators from $L^\infty(K,X)$ in $H^\infty(\Di,X)$ equipped with the operator norm.
 
 \begin{Th}\label{teo1.3}
 Suppose $K$ is a quasi-interpolating set of positive Lebesgue measure and $\zeta$ is an $\epsilon$-chain of $K$, $\epsilon\in (0,1)$, with respect to $\rho$ such that $\delta(\zeta)\ge\delta>0$. 
 
 Given $\nu\in (0,2-\sqrt 3]$ let
 \begin{equation}\label{eq1.8}
 \epsilon_\nu:=\frac{(2-\sqrt{3})^3}{6}\cdot\nu.
 \end{equation}
 There exist 
 
\begin{itemize}
\item[--] a natural number 
 \begin{equation}\label{eq1.9}
 k_\nu^*\le \frac{c}{\nu^2(1-\epsilon)}\cdot\max\left\{1,\frac{\log\frac{1}{\delta}}{(1-\epsilon_*)^2}\right\}
 \end{equation}
  for a numerical constant $c<5^3\cdot 10^5$;\smallskip
 
\item[--] interpolating sequences $\zeta_\nu^i\subset K$  with $\delta(\zeta_\nu^i)>\frac 1 2$, $1\le i\le k_\nu^*$, such that
\begin{equation}\label{eq1.10}
 K\subset \bigcup_{i=1}^{k_\nu^*} B_{\zeta_\nu^i}^{-1}(\bar\Di_{\epsilon_\nu})\subset \bigcup_{i=1}^{k_\nu^*} B_{\zeta_\nu^i}^{-1}(\Di_{6\epsilon_\nu})\subset [K]_\nu;
 \end{equation}

\item[--] holomorphic functions $H_\nu^i\in H^\infty(\Co\setminus  \bar{\Di}_{\epsilon_\nu},\mathscr B(K,X))$ vanishing at $\infty$ of norms
\begin{equation}\label{eq1.11}
\|H_\nu^i\|_\infty \le \frac{3}{5}\cdot \nu,\qquad 
1\le i\le k_\nu^*;
\end{equation}
 
\item[--] an operator $E_{\nu}^0\in \mathscr B(K,X)$
\end{itemize}  
such that for every $f\in L^\infty(K,X)$, $z\in \Di\setminus\overline{[K]}_\nu$
\begin{equation}\label{eq1.12}
 (L_K^Xf)(z)=(E_{\nu}^0 f)(z)+\sum_{i=1}^{k_\nu^*}\bigl(H_\nu^i(B_{\zeta_\nu^i}(z))f\bigr)(z).
 \end{equation}
 \end{Th}
%12461568
\begin{R}\label{rem1.4}
{\rm Let $R_{K;\nu}: C_\rho(\Di,X)\rightarrow C_\rho\bigl(\Di\setminus\overline{[K]}_\nu,X\bigr)$, $f\mapsto f|_{\Di\setminus\overline{[K]}_\nu}$, be the restriction operator. Expanding holomorphic functions $H_\nu^i$ in the Laurent series:
\[
H_\nu^i(w)=\sum_{j=1}^\infty E_{\nu, j}^i w^{-j},\quad w\in\Co\setminus\bar\Di_{\epsilon_\nu},
\]
for some  $E_{\nu, j}^i\in \mathscr B(K,X)$, $1\le i\le k_\nu^*$, $j\in\N$, we obtain from \eqref{eq1.12}:
\begin{equation}\label{eq1.13}
R_{K;\nu}\circ L_K^X=R_{K;\nu}\circ E_{\nu}^0+\sum_{i=1}^{k_\nu^*}\sum_{j=1}^\infty R_{K;\nu}\circ\bigl( E_{\nu, j}^i\cdot B_{\zeta_\nu^i}^{-j}\bigr),
\end{equation}
where all series converge uniformly on  $\Di\setminus\overline{[K]}_\nu$ due to  \eqref{eq1.10}. 
}
\end{R}
Let $\mathcal A:=\langle H^\infty,\bar H^\infty\rangle$ be the Banach subalgebra of the algebra of bounded complex-valued continuous functions on $\Di$ equipped with supremum norm generated by functions in $H^\infty$ and their complex conjugate. For $X=\Co$ we set $L^\infty(K):=L^\infty(K,\Co)$, $C_\rho(\Di):=C_\rho(\Di,\Co)$ and $L_K:=L_K^{\Co}$. Since $H^\infty\subset C_\rho(\Di)$, $\mathcal A$ is a closed subalgebra of the Banach algebra $C_\rho(\Di)$. 
As a corollary of Theorems \ref{teo1.1} and \ref{teo1.3} we obtain:
\begin{C}\label{cor1.5}
$L_K$ is a bounded linear operator from $L^\infty(K)$ into $\mathcal A$.
\end{C}
In particular, for every $f\in L^\infty(K)$ the function $L_K f$
 has radial limits a.\,e.  on the boundary $\mathbb S:=\{e^{i\theta}\in\Co\, :\, \theta\in\RR\}$ of $\Di$.
\begin{R}\label{rem1.6}
{\rm The algebra $\mathcal A$ is isometrically isomorphic via the Gelfand transform to the Banach algebra $C(\mathfrak M)$ of complex-valued continuous functions on the maximal ideal space $\mathfrak M$ of $H^\infty$. Recall that $\mathfrak M$ is the weak$^*$ compact subset of the unit ball of the dual space $(H^\infty)^*$  of nonzero complex homomorphisms of $H^\infty$.  The Gelfand transform $\hat{\,}:H^\infty\rightarrow C(\mathfrak M)$, $\hat{f}(\phi):=\phi(f)$, $\phi\in \mathfrak M$, $f\in H^\infty$, is an isometric monomorphism of Banach algebras.
The correspondence $\Di\ni z\mapsto\delta_z\in \mathfrak M$, where $\delta_z$ is the evaluation functional at $z\in\Di$, embeds $\Di$ as an open dense (due to the Carleson corona theorem \cite{C}) subset of $\mathfrak M$. 
}
\end{R}
The paper is organized as follows.

Section 2 contains some auxiliary results required for the proof of Theorem \ref{teo1.1}. Section 3 contains proofs of special cases of Theorems \ref{teo1.1} and \ref{teo1.3} for quasi-interpolating sets of small width. They are used in Section 4 to prove Theorems \ref{teo1.1} and \ref{teo1.3} for quasi-interpolating sets having chains of large characteristics. In turn, the latter results are used in Section 5 to prove Theorems \ref{teo1.1} and \ref{teo1.3} in the general case. Section 6 contains the proof  of Corollary \ref{cor1.5}. In Section 7, we describe  an application of Theorem \ref{teo1.1} to the Banach-valued corona problem for the algebra $H^\infty(\Di,A)$, where $A$ is a complex commutative unital Banach algebra.

 In a forthcoming paper, we present other applications of Theorems \ref{teo1.1} and \ref{teo1.3} to the theory of bounded Banach-valued holomorphic functions on $\Di$.

\sect{Auxiliary Results}
 \subsect{} 
 For $h\in L^\infty(\Di_s,X)$, $0<s\le 1$, we define
\begin{equation}\label{eq2.1}
(Eh)(z)=\frac{1}{2\pi i}\iint \limits_{\Di_s}\frac{h(w)}{w-z}dw\wedge d\bar w ,\qquad z\in\Di.
\end{equation}
Set
\begin{equation}\label{eq2.2}
\omega(t):=t\log\bigl( \mbox{$\frac{8}{t}$}\bigr),\quad 0< t\le 2.
\end{equation}
Let  $C^\omega(\Di,X)$ be the Banach space of bounded continuous $X$-valued functions $f$ on $\Di$ equipped with the norm
\[
\|f\|_{\omega}:=\max\{\|f\|_\infty, |f|_\omega\},\quad {\rm where}\quad |f|_\omega:=
 \sup_{z_1,z_2\in\Di,\,z_1\ne z_2}\frac{\|f(z_1)-f(z_2)\|_X}{\omega(|z_1-z_2|)}.
\]
\begin{Lm}\label{lem2.1}
$E$ is a bounded linear operator from $L^\infty(\Di_s,X)$ into $C^\omega(\Di,X)$ such that\\
for every $h\in L^\infty(\Di_s,X)$ the function $Eh$ is a weak solution of the equation 
\begin{equation}\label{eq2.3}
\frac{\partial H}{\partial\bar z}=h(z),\qquad z\in\Di,
\end{equation}
and $Eh|_{\Di\setminus\bar{\Di}_s}$ is a bounded $X$-valued holomorphic function.
\end{Lm}
\begin{proof}
Making the substitution $u:=w-z$, then passing to polar coordinates $u=re^{i\phi}$ and using the Fubini theorem we write \eqref{eq2.1} as follows
\begin{equation}\label{eq2.4}
(Eh)(z)=\frac{1}{2\pi i}\int_0^{2\pi}\left(\int_0^{1+s} h(re^{i\phi}+z)e^{-i\phi}dr\right)\, d\phi.
\end{equation}
This, the triangle inequality for integrals and the fact that $h=0$ a.e. on $\Co\setminus\Di_s$ imply
\begin{equation}\label{eq2.5}
\|Eh\|_\infty\le 2s\|h\|_\infty.
\end{equation}

Next, for $z_1,z_2\in\Di$ and $r_0:=\omega(|z_1-z_2|)$, $z_0:=\frac{z_1+z_2}{2}$
we get
\[
\begin{array}{l}
\displaystyle
 \|(Eh)(z_1)-(Eh)(z_2)\|_X\\
 \\
 \displaystyle =\frac{1}{2\pi}\left\| \iint \limits_{\ \ \Di_{z_0}(r_0)}\left(\frac{h(w)}{w-z_1}-\frac{h(w)}{w-z_2}\right) dw\wedge d\bar w +\iint \limits_{\ \ \Di\setminus \Di_{z_0}(r_0)}
 \frac{(z_2-z_1)h(w)}{(w-z_1)(w-z_2)}dw\wedge d\bar w\right\|_X\\
 \\
 \displaystyle \le\frac{1}{2\pi}\left\|\iint\limits_{\ \ \Di_{z_0}(r_0)}\frac{h(w)}{w-z_1}dw\wedge d\bar w\right\|_X+ 
\frac{1}{2\pi}\left\|\iint\limits_{\ \ \Di_{z_0}(r_0)} \frac{h(w)}{w-z_2} dw\wedge d\bar w\right\|_X \\
\\
\displaystyle +\frac{1}{2\pi}\left\|\iint\limits_{\ \ \Di\setminus \Di_{z_0}(r_0)}
 \frac{(z_2-z_1)h(w)}{(w-z_1)(w-z_2)}dw\wedge d\bar w\right\|_X=: I_1+I_2+I_3.
 \end{array}
\]
Integrals in $I_1$ and $I_2$ can be written in polar coordinates with respect to centers $z_1$ and $z_2$ similarly to \eqref{eq2.4}. These and triangle inequalities for integrals lead to the following estimates
\begin{equation}\label{eq2.6}
\max\{I_1,I_2\}\le \left(\frac{|z_1-z_2|}{2}+r_0\right)\cdot\|h\|_\infty.
\end{equation}
For the integral in $I_3$ we use polar coordinates with the center at $z_0$. Then using that
\[
r_0^2-\frac{|z_1-z_2|^2}{4}=|z_1-z_2|^2\left(\left(\log\left( \frac{8}{|z_1-z_2|}\right)\right)^2-\frac 14\right)\ge \frac 32 |z_1-z_2|^2,
\]
we obtain
\[
\begin{array}{l}
\displaystyle
I_3\le 
\|h\|_\infty\cdot \frac{|z_2-z_1|}{2\pi}\int_{0}^{2\pi}\left(\int_{r_0}^{2}\frac{1}{r^2-|\frac{z_1-z_2}{2}|^2}rdr\right)d\phi\\
\\
\displaystyle \quad\le  \frac 12\|h\|_\infty\cdot |z_2-z_1|\left(\log \left(4-\frac{|z_1-z_2|^2}{4}\right)-\log\left(r_0^2-\frac{|z_1-z_2|^2}{4}\right) \right)\\
\\
\displaystyle \quad\le \frac 12 \|h\|_\infty\cdot |z_2-z_1| \log\left( \frac{8}{3|z_1-z_2|^2}\ \right)\le \|h\|_\infty\cdot (r_0- \log 4 \cdot |z_2-z_1|).
\end{array}
\]
Combining this with \eqref{eq2.6} we get
\begin{equation}\label{eq2.7}
\|(Eh)(z_1)-(Eh)(z_2)\|_X\le 3r_0\cdot\|h\|_\infty.
\end{equation}
This and \eqref{eq2.5} complete the proof of the first statement of the lemma.

Next, the claim that $Eh$ is a weak solution of equation \eqref{eq2.3} can be proved in the same way as a similar statement in \cite[Ch.\,VIII.1]{Ga} for $X=\Co$.

Finally,  for $z\in\Di\setminus\bar{\Di}_s$
\[
(Eh)(z)=\frac{1}{2\pi i}\iint\limits_{\ \Di_s}\frac{h(w)}{z(\frac{w}{z}-1)}dw\wedge d\bar w=-\frac{1}{2\pi i}\sum_{i=1}^\infty \left(\iint\limits_{\ \Di_s}h(w) w^{i-1} dw\wedge d\bar w\right)z^{-i}.
\]
The series converges uniformly on compact subsets of $\Co\setminus \bar{\Di}_s$ to an $X$-valued holomorphic function. This proves the last statement of the lemma.
\end{proof}
\begin{R}\label{rem2.2}
{\rm Note that properties (i)-(iii) of Theorem \ref{teo1.1}  are valid for the operator $E$. Indeed, $TE=ET$ for a bounded linear operator of complex Banach spaces $T:X\rightarrow Y$ by the definition of the Bochner integral. 

Next, if $h\in L(\Di_s,X)$ has a compact essential range $R(h)$, then by the definition of a Bochner measurable essentially bounded function, $h$ is a limit a.e. on $\Di$ of a sequence $\{h_n\}_{n\in\N}$ of $X$-valued measurable simple functions with values in $R(h)$. Then due to \eqref{eq2.4} every $E(\frac{h_n}{2})$ has range in $\overline{\rm co}_\Co(R(h))$, the closure of the convex complex hull of $R(h)$, which is compact by the Mazur theorem. Since clearly $\{(Eh_n)(z)\}_{n\in\N}$ converges to $(Eh)(z)$ for all $z\in\Di$, the range of $Eh$ is contained in the compact set $2\cdot\overline{\rm co}_\Co(R(h))\subset X$, as required. Hereafter  for $S_1,S_2\in X$, $\lambda\in\Co$, 
\[
\lambda\cdot S_1:=\{\lambda x\in X\, :\, x\in S_1\},\qquad S_1+S_2:=\{x_1+x_2\in X\, :\, x_i\in S_i,\, i=1,2\}.
\]

Finally, if $h$ is continuously differentiable on an open set $U$, then it easily follows from equation \eqref{eq2.4} that
 $E(\rho h)$ is continuously differentiable on $\Di$ for every $C^\infty$ function $\rho$ with compact support in $U$ equals $1$ on an open subset $V$ of $U$. Also, since $g_V:=Eh|_V-E(\rho h)|_V$ is a weak solution of the equation  $\frac{\partial F}{\partial\bar z}=0$, it is a bounded $X$-valued holomorphic function on $V$. Then $Eh|_V=E(\rho h)|_V+g_V$ is continuously differentiable on $V$. This gives the required statement.
}
\end{R}
%=== 
 \subsect{}
 We require the following result.
 \begin{Lm}[\mbox{\cite[Ch.\,X,\,Lm.\,1.4,\, Ch.\,VII,\,Lm.\,5.3]{Ga}}]\label{lem2.3}
Let $B_\zeta$ be the interpolating Blaschke product with zeros $\zeta=\{z_n\}$  such that
\[
\delta(\zeta)=\inf_n\, (1-|z_n|^2)|B_\zeta'(z_n)|\ge\delta>0.
\]
Suppose $\lambda\in (0,1)$ and $r:=r(\lambda)\in (0,1)$ satisfy
\[
\frac{2\lambda}{1+\lambda^2}<\delta\quad{\rm and}\quad r=\frac{\delta-\lambda}{1-\lambda\delta}\lambda.
\]
Then 
\begin{itemize}
\item[(i)]
$B_\zeta^{-1}(\Di_r)=\{z\in\Co\, :\, |B_\zeta(z)|<r\}$ is the union of pairwise disjoint domains $V_{\zeta n}\,(\ni z_n)$ such that 
\[
V_{\zeta n}\subset D(z_n,\lambda);\smallskip
\]
\item[(ii)]
$B_\zeta$ maps every $V_{\zeta n}$ biholomorpically onto $\Di_r$;\smallskip
\item[(iii)] 
Every sequence $\omega=\{ w_n\}$ with $w_n\in D(z_n,\lambda)$ for all $n$,
is  interpolating for $H^\infty$ and
 \[
\delta(\omega) \ge\frac{\delta-\frac{2\lambda}{1+\lambda^2}}{1-\frac{2\delta \lambda}{1+\lambda^2}}.
\]
%Moreover,  for $|w|<r$ the function
%\[
%B_{\zeta w}(z)=\frac{B_\zeta(z)-w}{1-\bar{w}B_\zeta(z)},\qquad z\in\Di,
%\]
%is the product of an interpolating Blaschke product with one zero in each $V_{\zeta n}$ and a constant of modulus $1$.
\end{itemize}
\end{Lm}
By $b_{\zeta n} :\Di_r\to V_{\zeta n}$ we denote the (holomorphic) inverse of $B_\zeta|_{V_n}$.
Since $B_\zeta(z_n)=0$ and $\|B_\zeta\|_{\infty}=1$,  by the Schwarz-Pick theorem $D(z_n,r)\subset b_{\zeta n}(\Di_r)\subset V_{\zeta n}\, (\subset D(z_n,\lambda))$.

Let $M_\zeta$ be the constant of interpolation of $\zeta$ and $N$ be the set of indices of $\{z_n\}$.
\begin{Prop}[\mbox{\cite[Prop.\,3.7]{Br2}}]\label{prop2.4}
There exist functions $f_j\in H^\infty(\Di\times\Di_{\frac{r}{3M_\zeta}})$, $j\in N$, such that
\begin{equation}\label{eq2.8}
f_j(b_{\zeta j}(w),w)=1,\quad f_j(b_{\zeta k}(w),w)=0,\quad k\ne j,
\end{equation}
\begin{equation}\label{eq2.9}
\sum_{j\in N} |f_j(z,w)|\le 2M_\zeta,\quad (z,w)\in \Di\times\Di_{\frac{r}{3M_\zeta}}.
\end{equation}
\end{Prop}
\begin{proof}
Due to \cite[Ch.~VII, Th.~2.1]{Ga} there exist functions $g_j\in H^\infty$, $j\in N$, such that
\[
g_j(z_j)=1,\quad g_j(z_k)=0,\quad k\ne j,\quad\text{and}\quad \sum_{j\in N} |g_j(z)|\le M_\zeta,\quad z\in \Di .
\]
Consider a bounded linear operator $L:\ell^\infty(N)\to H^\infty$ of norm $\|L\|=M_\zeta$
\[
L(a)(z):=\sum_{j\in N} a_j g_j(z),\quad a=\{a_j\}_{j\in N},\quad z\in\Di.
\]
Let $R(w): H^\infty\to \ell^\infty(N)$ be the restriction operator to $b(w)=\{b_j(w)\}_{j\in N}$, $w\in\Di_r$. Then
\[
(R(w)\circ L)(a):=\left\{\sum_{j\in N} a_j g_j(b_k(w))\right\}_{k\in N}.
\]
This and the Cauchy estimates for derivatives of bounded holomorphic functions imply that $P(w):=R(w)\circ L:\ell^\infty(N)\to\ell^\infty(N)$, $w\in\Di_r$, is a family of bounded linear operators  of norms $\le M_\zeta$ holomorphically depending on $w$ and such that $P(0)={\rm id}$. The Cauchy estimates yield $\|\frac{dP}{dw}(w)\|\le\frac{M_\zeta}{r-|w|}$. Thus for $|w|\le\frac{r}{3M_\zeta}$ we have
\[
\bigl|\|P(w)\|-1\bigr|:=\bigl|\|P(w)\|-\|P(0)\|\bigr|\le |w|\frac{M_\zeta}{r-|w|}\le\frac{1}{2}.
\]
In particular, the operator $P(w)$ is invertible and $\|P(w)^{-1}\|\le 2$. 

Set
\[
\hat L(w):=L\circ P(w)^{-1},\quad w\in\Di_{\frac{r}{3M_\zeta}} .
\]
Then the linear operator $\hat L(w):\ell^\infty(N)\to H^\infty$ is continuous, holomorphically depends on $w\in\Di_{\frac{r}{3M_\zeta}}$ and $\|\hat L(w)\|\le 2M_\zeta$. Moreover, $R(w)\circ\hat L(w)={\rm id}$.

Finally, define
\begin{equation}\label{eq2.10}
f_j(\cdot,w):=\hat L(w)(\delta_j),\quad j\in N;
\end{equation}
here $\delta_j=\{\delta_{ij}\}_{i\in N}\in \ell^\infty(N)$, $\delta_{ij}=1$ if $i=j$ and $0$ otherwise. Clearly, functions $f_j\in H^\infty(\Di\times\Di_{\frac{r}{3M_\zeta}})$, $j\in N$, satisfy \eqref{eq2.8}, \eqref{eq2.9}.
\end{proof}
Let us consider the open annulus
\begin{equation}\label{eq2.11}
A:=\Di_{\frac{r}{3M_\zeta}}\setminus \bar{\Di}_{\frac{r}{6M_\zeta}}
\end{equation}
and set $A_\zeta:=B_\zeta^{-1}(A)\subset\Di$.
By the definition, see Lemma \ref{lem2.3},
\[
A_\zeta:=\bigsqcup_{n} A_{\zeta n},\quad {\rm where}\quad
A_{\zeta n}:=V_{\zeta n}\cap A_\zeta.
\]
Moreover, $B_\zeta$ maps every $A_{\zeta n}$ biholomorphically onto $A$. We also set
\[
S_{\zeta 1}:=B_\zeta^{-1}\bigl(\Di\setminus  \bar{\Di}_{\frac{r}{6M_\zeta}}\bigr),\quad S_{\zeta 2}:=B_\zeta^{-1}(\Di_{\frac{r}{3M_\zeta}}).
\]
Let $X$ be a complex Banach space. In what follows, for a complex manifold $U$ we denote by $H^\infty(U,X)$ the Banach space of bounded $X$-valued holomorphic functions $f$ on $U$ equipped with norm $\|f\|_\infty:=\sup_{z\in U}\|f(z)\|_X$.
\begin{Th}\label{teo2.5}
There exist bounded linear operators $T_i: H^\infty(A_\zeta,X)\rightarrow H^\infty(S_{\zeta i},X)$, $i=1,2$, of norms $\|T_1\|\le 6M_\zeta$, $\|T_2\|\le 4M_\zeta $ such that
\begin{equation}\label{eq2.12}
(T_1 f)|_{A_\zeta}+(T_2 f)|_{A_\zeta}=f\quad {\rm for\ all}\quad f\in H^\infty(A_\zeta,X).
\end{equation}
\end{Th}
\begin{proof}
Let $\Gamma_{B_\zeta}(A):=\{(z,w)\in A_\zeta\times A\, :\, w=B_\zeta(z)\}$ be the graph of $B_\zeta|_{A_{\zeta}}$,
and let $R_{A}: H^\infty(\Di\times A,X)\rightarrow H^\infty(\Gamma_{B_\zeta}(A),X)$, 
$f\mapsto f|_{\Gamma_{B_\zeta}(A)}$, be the restriction operator.
\begin{Lm}\label{lem2.6}
There exists a bounded linear operator $S_A^X: H^\infty(A\times N , X)\to H^\infty(\Di\times A ,X)$ of norm $\le 2M_\zeta$  such that
\[
(R_{A}\circ S_A^X)(g)(b_j(w),w)=g(w,j)\quad\text{for all}\quad (w,j)\in A\times N,\ g\in H^\infty(A\times N , X).
\]
\end{Lm}
\begin{proof}
We define
\begin{equation}\label{eq2.13}
S_A^X(g)(z,w):=\sum_{j\in N} f_{j}(z,w)g(w,j)
\end{equation}
with $f_j$ as in Proposition \ref{prop2.4}. Then the required result follows from \eqref{eq2.8} and \eqref{eq2.9}.
\end{proof}

Next, consider an isometric isomorphism $I: H^\infty(A_\zeta,X)\rightarrow H^\infty(A\times N,X)$,
\begin{equation}\label{eq2.14}
(If)(w,j):=(f\circ b_{\zeta j})(w),\quad (w,j)\in A\times N,\quad f\in H^\infty(A_\zeta,X).
\end{equation}
Then the linear operator $S_A^X\circ I:H^\infty(A_\zeta,X)\rightarrow H^\infty(\Di\times A,X)$ of norm $\le 2M_\zeta$ is given by
\begin{equation}\label{eq2.15}
((S_A^X\circ I)f)(z,w)=\sum_{j\in N} f_{j}(z,w)(f\circ b_{\zeta j})(w),\quad (z,w)\in\Di\times A,\quad f\in H^\infty(A_\zeta,X).
\end{equation}
Let
\begin{equation}\label{eq2.16}
\begin{array}{l}
\displaystyle
((S_A^X\circ I)f)(z,w)=\sum_{n=-\infty}^{-1} a_{n}(z)w^{n}+\sum_{n=0}^\infty
a_{n}(z)w^{n}=: g_1(z,w)+g_2(z,w),\\
\\
\displaystyle \hspace*{18mm}
a_{n}(z):=\frac{1}{2\pi i}\oint_{|\xi|=\frac{r}{4M_\zeta}} \frac{((S_A^X\circ I)f)(z,\xi)}{\xi^{n+1}}\,d\xi,\qquad n\in\Z,\ z\in\Di,
\end{array}
\end{equation}
be the Laurent series expansion of $(S_A^X\circ I)f$ with respect to $w\in A$.

By the Cauchy formula for a sufficiently small $\varepsilon> 0$ we get\smallskip
\begin{equation}\label{eq2.17}
\begin{array}{l}
\displaystyle
((S_A^X\circ I)f)(z,w)\\
\\
\displaystyle =\frac{1}{2\pi i}\oint_{|\xi|=\frac{r}{6M_\zeta}+\varepsilon} \frac{((S_A^X\circ I)f)(z,\xi)}{w-\xi}\,d\xi+\frac{1}{2\pi i}\oint_{|\xi|=\frac{r}{3M_\zeta}-\varepsilon}  \frac{((S_A^X\circ I)f)(z,\xi)}{\xi-w}\,d\xi\\
\\
\displaystyle
 = g_{1}(z,w)+g_{2}(z,w),\qquad z\in\Di,\quad \mbox{$\frac{r}{6M_\zeta}+\varepsilon\le |w|\le \frac{r}{3M_\zeta}-\varepsilon$}.
\end{array} 
\end{equation}
 Formula \eqref{eq2.17} shows that $g_2$ extends to a function  $\tilde g_2\in H^\infty(\Di\times \Di_{\frac{r}{3M_\zeta}},X)$  of norm 
\begin{equation}\label{eq2.18}
\|\tilde g_2\|_\infty\le \|(S_A^X\circ I)f\|_\infty+ \lim_{\varepsilon\to 0}\frac{\|(S_A^X\circ I)f\|_\infty}{\frac{r}{3M_\zeta}-\frac{r}{6M_\zeta}-2\varepsilon}\cdot \left(\frac{r}{6M_\zeta}+\varepsilon\right)\le 4M_\zeta\|f\|_\infty.
\end{equation}
Similarly, $g_1$ extends to a function $\tilde g_1\in H^\infty(\Di\times \Co\setminus\bar{\Di}_{\frac{r}{6M_\zeta}},X)$ having zero at $\infty$ of norm 
 \begin{equation}\label{eq2.19}
 \|\tilde g_1\|_\infty\le \|(S_A^X\circ I)f\|_\infty+\|\tilde g_2\|_\infty\le 6M_\zeta\|f\|_\infty.
 \end{equation}

Finally, we set
\begin{equation}\label{eq2.20}
T_1 f(z):=\tilde g_1(z,B_\zeta(z)),\quad z\in S_{\zeta 1},\quad {\rm and}\quad
T_2 f(z):=\tilde g_2(z,B_\zeta(z)),\quad z\in S_{\zeta 2}.
\end{equation}
Then $T_i: H^\infty(A_\zeta,X)\rightarrow H^\infty(S_{\zeta i},X)$, $i=1,2$, are bounded linear operators of norms $\|T_1\|\le 6M_\zeta$, $\|T_2\|\le 4M_\zeta $ and due to Lemma \ref{lem2.6} and \eqref{eq2.15}, \eqref{eq2.16} one obtains for $z\in A_{\zeta j}$, $j\in N$, 
\[
(T_1 f)(z)+(T_2 f)(z)=((S\circ I)f)(z,B_\zeta(z))=(If)(B_\zeta(z),j)=(f\circ b_{\zeta j})(B_\zeta(z))=f(z).
\]
This completes the proof of the theorem.
\end{proof}
\begin{R}\label{rem2.7}
{\rm Equation \eqref{eq2.16} shows that
\begin{equation}\label{eq2.21}
\begin{array}{l}
\displaystyle
T_1f=\sum_{n=-\infty}^{-1} a_{n}B_\zeta^{n},\qquad T_2f=\sum_{n=0}^\infty
a_{n}B_\zeta^{n},\quad {\rm where}\\
\\
\displaystyle
a_{n}(z):=\frac{1}{2\pi i}\oint_{|\xi|=\frac{r}{4M_\zeta}} \frac{((S_A^X\circ I)f)(z,\xi)}{\xi^{n+1}}\,d\xi,\qquad n\in\Z,\ z\in\Di.
\end{array}
\end{equation} 
Here all $a_{n}\in H^\infty(\Di,X)$ and the series for $T_1f$ and $T_2f$ converge uniformly on subsets $B_\zeta^{-1}(U)$, where $U$ are compact subsets of $\Di\setminus  \bar{\Di}_{\frac{r}{6M_\zeta}}$ and $\Di_{\frac{r}{3M_\zeta}}$, respectively.

Moreover, if the range ${\rm im}(f)$ of $f$ is a relatively compact subset of $X$, then due to \eqref{eq2.15}, ${\rm im}((S_A^X\circ I)f)\subset 2M_\zeta\cdot\overline{{\rm co}}_\Co({\rm im}(f))$. In turn, due to \eqref{eq2.17} arguing as in \eqref{eq2.18}, \eqref{eq2.19} we obtain
\[
\begin{array}{l}
\displaystyle
{\rm im}(\tilde g_2)\subset {\rm im}((S_A^X\circ I)f)+\frac{\frac{r}{6M_\zeta}}{\frac{r}{3M_\zeta}-\frac{r}{6M_\zeta}}\cdot\overline{{\rm co}}_\Co({\rm im}((S_A^X\circ I)f))\subset 2\cdot \overline{{\rm co}}_\Co({\rm im}((S_A^X\circ I)f)),\\
\\
\displaystyle
{\rm im}(\tilde g_1)\subset  {\rm im}((S_A^X\circ I)f)+{\rm im}(\tilde g_2)\subset 3\cdot \overline{{\rm co}}_\Co({\rm im}((S_A^X\circ I)f)).
\end{array}
\]
Hence,
${\rm im}(T_if)\subset 6\cdot\overline{{\rm co}}_\Co((S_A^X\circ I)f)$, $i=1,2$. These implications and the classical Mazur theorem show that ${\rm im}(T_if)\subset X$, $i=1,2$, are relatively compact.

Finally, note that $LT_i=T_iL$, $i=1,2$, for a bounded linear operator of complex Banach spaces $L:X\rightarrow Y$ by \eqref{eq2.13} and the uniqueness of Laurent series expansions.
}
\end{R}
\sect{Theorems \ref{teo1.1} and \ref{teo1.3} for Sets of Small Width}
We retain notation of Lemma \ref{lem2.3}. In this section we prove Theorems \ref{teo1.1} and \ref{teo1.3} (see Theorem \ref{teo3.3} and Corollary \ref{cor3.4}) for a Lebesgue measurable set $K\subseteq B_\zeta^{-1}(\Di_\frac{r}{6M_\zeta})$,
where $\zeta=\{z_n\}_{n\in N}$ is an $\epsilon$-chain with respect to $\rho$ such that $\delta(\zeta)\ge\delta>0$. In this case,
\begin{equation}\label{eq3.1}
K=\bigsqcup_{n} K_n,\quad {\rm where}\quad K_n:=V_{\zeta n}\cap K,\ \ n\in N,
\end{equation}
and, moreover, since $V_{\zeta n}\subset D(z_n,\lambda)$
(see Lemma \ref{lem2.3}), $b_{\zeta n}$ maps $\Di_r$ into $D(z_n,\lambda)$ for all $n\in N$. Then applying to $b_{\zeta n}$ the Schwarz-Pick theorem we obtain
\begin{equation}\label{eq3.2}
D(z_n,\mbox{$\frac{r}{6M_\zeta}$})\subset b_{\zeta n}(\Di_{\frac{r}{6M_\zeta}})\subset D(z_n,\mbox{$\frac{\lambda}{6M_\zeta}$}).
\end{equation}

Therefore,  since $K_n\subset  b_{\zeta n}(\Di_{\frac{r}{6M_\zeta}})$ and $\frac{\lambda}{6M_\zeta}<\frac{\delta}{6}<\delta(\zeta)$ while $\rho(z_i,z_j)>\delta(\zeta)$ for all $i\ne j$ (if $N$ contains at least two elements),
 \begin{equation}\label{eq3.3}
 K_n\subset\left\{
 \begin{array}{ccc}
 D(z_n, \epsilon)&{\rm if}&\epsilon\le \frac{r}{6M_\zeta}\\
 D(z_n,\frac{\lambda}{6M_\zeta})&{\rm if}&\epsilon> \frac{r}{6M_\zeta},
\end{array}
\right.   \qquad n\in N.
\end{equation}
We set
\begin{equation}\label{eq3.4}
c(\epsilon):=\left\{
 \begin{array}{ccc}
 \epsilon&{\rm if}&\epsilon\le \frac{r}{6M_\zeta}\\
 \frac{\lambda}{6M_\zeta}&{\rm if}&\epsilon> \frac{r}{6M_\zeta}.
\end{array}
\right.
\end{equation}

Recall that $C_\rho(S,X)$, $S\subset\Di$, stands for the Banach space of bounded continuous functions $f:S\rightarrow X$ uniformly continuous with respect to $\rho$ equipped with norm $\|f\|_\infty=\sup_{z\in S}\|f(z)\|_X$.

First, we prove the following result.
\begin{Lm}\label{lem3.1}
There is a bounded linear operator $E_{K}^X: L^\infty(K,X)\rightarrow C_\rho(B_\zeta^{-1}(\Di_r),X)$ of norm 
\[
\|E_{K}^X\|\le \frac{2c(\epsilon)}{1-c(\epsilon)^2}
\] 
such that for
every $f\in L^\infty(K,X)$ the $X$-valued function $E_{K}^Xf$ is a
weak solution of the equation
\begin{equation}\label{eq3.5}
\frac{\partial F}{\partial\bar{z}}=\frac{f(z)}{1-|z|^2},\quad z\in B_\zeta^{-1}(\Di_r).
\end{equation}
\end{Lm}
\begin{proof}
Recall that $\zeta=\{z_n\}_{n\in N}$ and $D(z_n,r)\subset V_{\zeta n}\subset D(z_n,\lambda)$, see Lemma \ref{lem2.3}.  For $f\in L^\infty(K,X)$, we set $f_n:=f|_{V_{\zeta n}}\in L^\infty(K_n,X)$.  Let $g_n(w):=\frac{w+z_n}{1+\bar{z}_n w}$, $w\in\Di$, be the M\"{o}bius transformation of $\Di$ which maps $\Di_\lambda$ biholomorphically onto $D(z_n,\lambda)$.
Making the substitution $z=g_n(w)$  we obtain
\begin{equation}\label{eq3.6}
\frac{f_n(z)}{1-|z|^2}\,d\bar z=\frac{(f_n\circ g_n)(w)}{1-|g_n(w)|^2}\,\overline{g_n'(w)}\, d\bar w=\frac{1+\bar{z}_n w}{1+z_n\bar w}\cdot\frac{(f_n\circ g_n)(w)}{1-|w|^2}\, d\bar w=:\tilde f_n(w)\, d\bar w.
\end{equation}
Here $\tilde f_n\in L^\infty(\tilde K_n, X)$,  $\tilde K_n:=g_n^{-1}(K_n)\subset\Di_{c(\epsilon)}\, (\subset\Di_{\frac{\lambda}{6M_\zeta}})$ and
\[
\|\tilde f_n\|_\infty\le \frac{\|f_n\|_\infty}{1-c(\epsilon)^2}.
\]
Applying to $\tilde f_n$ the linear operator $E$ of Lemma 
\ref{lem2.1} we obtain  that $E\tilde f_n$ is a weak solution of the equation $\frac{\partial  F}{\partial\bar{w}}=\tilde f_n(w)$,  $w\in\Di$, such that, see \eqref{eq2.5}, \eqref{eq2.7},
\begin{equation}\label{eq3.7}
\|E\tilde f_n\|_\infty\le 2c(\epsilon)\|\tilde f_n\|_\infty\le  \frac{2c(\epsilon)\|f_n\|_\infty}{1-c(\epsilon)^2}\quad {\rm and}\quad |E\tilde f_n|_\omega\le 3\|\tilde f_n\|_\infty\le\frac{3\|f\|_\infty}{1-c(\epsilon)^2} .
\end{equation}
We define
\begin{equation}\label{eq3.8}
(E_K^X f)(z):=(E\tilde f_n)(g_n^{-1}(z)),\quad z\in V_{\zeta n},\quad n\in N.
\end{equation}
Then due to \eqref{eq3.6}, $E_K^X f$ is a weak solution of equation \eqref{eq3.5}. 

Since due to \eqref{eq3.7} the sequence
$\{E\tilde f_n\}_{n\in N}$ is equicontinuous on $\Di$ with respect to the Euclidean metric $d$ and  $d\le\rho$, it is equicontinuous on $\Di$ with respect to $\rho$ as well. Since each $g_n$ is an isomorphic isometry of the metric space $(\Di,\rho)$, the latter implies  (recall that $B_\zeta^{-1}(\Di_r)=\sqcup_{n\in N}\,V_{\zeta n}$)
\begin{equation}\label{eq3.9}
\sup_{z\in B_\zeta^{-1}(\Di_r)}|(E_K^X f)(z)|\le \frac{2c(\epsilon)\|f_n\|_\infty}{1-c(\epsilon)^2},
\end{equation}
and given $\varepsilon>0$ there is $\delta_{\varepsilon}\in (0,1)$ such that for all $z_1,z_2\in V_{\zeta\,n}$, $n\in N$, satisfying $\rho(z_1,z_2)<\delta_\varepsilon$,
\begin{equation}\label{eq3.10}
|(E_K^X f)(z_1)-(E_K^X f)(z_2)|<\varepsilon.
\end{equation}

Moreover, according to Lemma \ref{lem2.3} (see also the proof of \cite[Ch.\,X,\,Lm.\,1.4]{Ga})
we have for $N$ containing at least two elements, $n_1,n_2\in N$, $n_1\ne n_2$, 
\begin{equation}\label{eq3.11}
{\rm dist}_\rho(V_{\zeta n_1}, V_{\zeta n_2})=\inf_{y_i\in V_{\zeta n_i},\, i=1,2}\rho(y_1,y_2)\ge\delta-\frac{2\lambda}{1+\lambda^2}>0.
\end{equation}
Then for $\tilde\delta_\varepsilon:=\min\left\{\delta_\varepsilon ,\delta-\frac{2\lambda}{1+\lambda^2}\right\}$ we get from \eqref{eq3.9}--\eqref{eq3.11} for all $z_1,z_2\in B_\zeta^{-1}(\Di_r)$, $\rho(z_1,z_2)<\tilde\delta_\varepsilon$,
\[
|(E_K^X f)(z_1)-(E_K^X f)(z_2)|<\varepsilon.
\]
This and \eqref{eq3.9} imply that $E_K^Xf\in C_\rho(B_\zeta^{-1}(\Di_r),X)$ and $\|E_K^X\|\le \frac{2c(\epsilon)}{1-c(\epsilon)^2}$, as required.
\end{proof}
\begin{R}\label{rem3.2}
{\rm Since analogs of properties (i)--(iii) of the operator of Theorem \ref{teo1.1} are valid for the operator $E$, see Remark \ref{rem2.2}, such properties are valid for the operator $E_K^X$ as well.
}
\end{R}

Suppose $K\subset B_\zeta^{-1}(\Di_{\frac{r}{6M_\zeta}})$ is a Lebesgue measurable quasi-interpolating set having an $\epsilon$-chain $\zeta=\{z_j\}_{j\in N}$ with respect to $\rho$ of the characteristic  $\delta(\zeta)\ge\delta>0$.
\begin{Th}\label{teo3.3}
There is  a bounded linear operator $L_K^X: L^\infty(K,X)\to C_\rho(\Di,X)$ satisfying conditions (i)--(iii) of Theorem \ref{teo1.1} of norm 
\[
\|L_K^X\|\le\frac{12c(\epsilon) M_\zeta}{1-c(\epsilon)^2}
\]
such that for every $f\in L^\infty(K,X)$ the function $L_K^X f$ is a weak solution of equation \eqref{eq1.1}.
\end{Th}
\begin{proof}
Let $R_{A_\zeta}: C_\rho(B_\zeta^{-1}(\Di_r), X)\rightarrow C_\rho(A_\zeta, X)$, $f\mapsto f|_{A_\zeta}$, $A_\zeta=B_\zeta^{-1}(A)$, see \eqref{eq2.11}, be the restriction operator. The composite operator $R_{A_\zeta}\circ E_K^X$ maps $L^\infty(K,X)$ into $H^\infty(A_\zeta,X)$ 
(because $K\subset B_\zeta^{-1}(\Di_{\frac{r}{6M_\zeta}})$)
and therefore Theorem \ref{teo2.5} can be applied. 

We set
\begin{equation}\label{eq3.12}
(L_K^Xf)(z):=\left\{
\begin{array}{ccc}
\bigl((E_K^X- T_2\circ R_{A_\zeta}\circ E_K^X)f\bigr)(z)&{\rm if}&z\in S_{\zeta 2}=B_\zeta^{-1}(\Di_{\frac{r}{3M_\zeta}})\\
\\
\bigl((T_1\circ R_{A_\zeta}\circ E_K^X)f\bigr)(z)&{\rm if}&z\in S_{\zeta 1}=B_\zeta^{-1}\bigl(\Di\setminus  \bar{\Di}_{\frac{r}{6M_\zeta}}\bigr).
\end{array}
\right.
\end{equation}
Since $\Di=S_{\zeta 1}\cup S_{\zeta 2}$, $A_\zeta=S_{\zeta 1}\cap S_{\zeta 2}$ and since due to Theorem \ref{teo2.5}
\[
\bigl((E_K^X- T_2\circ R_{A_\zeta}\circ E_K^X)f\bigr)(z)=\bigl((T_1\circ R_{A_\zeta}\circ E_K^X)f\bigr)(z),\quad z\in A_\zeta,
\]
the operator $L_K^X$ is well defined. Since functions
$(T_i\circ R_{A_\zeta}\circ E_K^X)f$ are holomorphic, the function $L_K^Xf$  is a weak solution of \eqref{eq1.1} by Lemma \ref{lem3.1}. Moreover, 
\[
\begin{array}{l}
\displaystyle
\|L_K^Xf\|_\infty\le \max\bigl\{\|(E_K^X- T_2\circ R_{A_\zeta}\circ E_K^X)f\|_\infty,\|(T_1\circ R_{A_\zeta}\circ E_K^X)f\|_\infty\bigr\}\\
\\
\displaystyle
\le \max\bigl\{1+\|T_2\|\cdot\|R_{A_\zeta}\|,\|T_1\|\cdot\|R_{A_\zeta}\| \bigr\}\cdot\|E_K^X\|\cdot\|f\|_\infty\le 6M_\zeta\cdot\frac{2c(\epsilon)}{1-c(\epsilon)^2}\cdot\|f\|_\infty.
\end{array}
\]
Hence, $\|L_K^X\|\le\frac{12c(\epsilon) M_\zeta}{1-c(\epsilon)^2}$, as required.

Next, let us prove that $L_K^X f\in C_\rho(\Di,X)$.
 
 To this end, consider the open cover
\[
\Di=B_\zeta^{-1}\bigl(\Di_{\frac{r}{2M_\zeta}}\bigr)\cup B_\zeta^{-1}\bigl(\Di\setminus\bar{\Di}_{\frac{r}{3M_\zeta}}\bigr)=:W_1\cup W_2.
\]
Since $E_K^Xf$ is holomorphic on $B_\zeta^{-1}(\Di_r)\setminus \bar{K}\supset B_\zeta^{-1}\bigl(\Di_r\setminus\bar{\Di}_{\frac{r}{6M_\zeta}}\bigr)
\supset A_\zeta$,  formula \eqref{eq3.12} shows that
$(T_2\circ R_{A_\zeta}\circ E_K^X)f$ extends to a bounded $X$-valued holomorphic function, say $g$, on  $B_\zeta^{-1}(\Di_r)$.  We define a holomorphic function $G\in H^\infty(\Di_r\times N,X)$ by the formula (cf. \eqref{eq2.14})
\[
G(w,j):=(g\circ b_{\zeta j})(w),\quad (w,j)\in \Di_r\times N.
\]
Expanding $G$ in the Maclaurin series in $w\in\Di_r$ and then substituting $w=B_\zeta(z)$ we obtain (as $b_{\zeta j}(B_\zeta(z))=z$, $z\in V_{\zeta j}$, $j\in N$)
\begin{equation}\label{eq3.13}
g(z)=\sum_{i=1}^\infty c_i(z) B_\zeta^i(z),\qquad z\in B_\zeta^{-1}(\Di_r),
\end{equation}
where $c_i\in H^\infty(B_\zeta^{-1}(\Di_r),X)$ are locally constant functions and the series converges uniformly on subsets $B_\zeta^{-1}(\Di_t)$, $t<r$, and, in particular, on $W_1$. Due to inequality \eqref{eq3.11}, all $c_i\in C_\rho(B_\zeta^{-1}(\Di_r),X)$. Also, 
$B_\zeta\in C_\rho(\Di,\Co)$. Hence, uniform convergence of  series  \eqref{eq3.13} on $W_1$ implies that $g|_{W_1}\in C_\rho(W_1,X)$. Thus, since $E_K^Xf\in C_\rho(B_\zeta^{-1}(\Di_r),X)$ (see Lemma \ref{lem3.1}),
$L_K^X f|_{W_1}= E_K^Xf|_{W_1}+g|_{W_1}\in C_\rho(W_1,X)$ as well.

Further, equation \eqref{eq2.21} shows that $L_K^Xf|_{S_{\zeta 1}}:=(T_1\circ R_{A_\zeta}\circ E_K^X)f|_{S_{\zeta 1}}$ can be expanded in a series  $\sum_{n=1}^{\infty} d_{n}B_\zeta^{-n}$ for some $d_n\in H^\infty(\Di,X)$ converging uniformly on subsets $B_\zeta^{-1}(\Di\setminus\bar{\Di}_t)$,  $t\in (\frac{r}{6M_\zeta},1)$, and, in particular, on $W_2$. Since, all $d_n\in C_\rho(\Di,X)$ and $\frac{1}{B_\zeta}\in C_\rho(W_2,\Co)$,  uniform convergence of the series implies that $L_K^Xf|_{W_2}=(T_1\circ R_{A_\zeta}\circ E_K^X)f|_{W_2}\in C_\rho(W_2,X)$.

Now, every open disk $\Di_{\frac{r}{12M_\zeta}}(w)\subset\Di$ is a subset of either $\Di_{\frac{r}{2M_\zeta}}$ or $\Di\setminus\bar{\Di}_{\frac{r}{3M_\zeta}}$. Moreover, if $z\in\Di$ is such that $w=B_\zeta(z)$, then 
$B_\zeta\bigl(D(z,\frac{r}{12M_\zeta})\bigr)\subset \Di_{\frac{r}{12M_\zeta}}(w)$ (due to the Schwarz--Pick theorem). Hence, every open pseudohyperbolic disk
$D(z,\frac{r}{12M_\zeta})$ is a subset of either $W_1$ or $W_2$. Since $L_K^X f$ is uniformly continuous with respect to $\rho$ on $W_i$, $i=1,2$, the latter implies that $L_K^X f\in C_\rho(\Di,X)$, as stated.

Finally, statements (i)--(iii) of Theorem \ref{teo1.1} follow from Remarks \ref{rem2.2}, \ref{rem2.7} and \ref{rem3.2}.

The proof of the theorem is complete.
 \end{proof}
 The next result describes the composition of the operator $L_K^X$ of Theorem \ref{teo3.3} with the restriction operator to $B_\zeta^{-1}\bigl(\Di\setminus  \bar{\Di}_{\frac{r}{6M_\zeta}}\bigr)$ (see Theorem \ref{teo1.3}). Recall that $\mathscr B(K,X)$ stands for the Banach space of bounded linear operators from $L^\infty(K,X)$ in $H^\infty(\Di,X)$ equipped with the operator norm.
 \begin{C}\label{cor3.4}
 There exists a bounded holomorphic function $H:\Co\setminus  \bar{\Di}_{\frac{r}{6M_\zeta}}\rightarrow\mathscr B(K,X)$ vanishing at $\infty$  of norm
 \begin{equation}\label{eq3.14}
 \|H\|_\infty \le \frac{12c(\epsilon) M_\zeta}{1-c(\epsilon)^2}
 \end{equation}
 such that 
 \begin{equation}\label{eq3.15}
 (L_K^Xf)(z)=\bigl(H(B_\zeta(z))f\bigr)(z),\quad z\in B_\zeta^{-1}\bigl(\Di\setminus  \bar{\Di}_{\frac{r}{6M_\zeta}}\bigr),\quad f\in L^\infty(K,X).
 \end{equation}
 \end{C}
 \begin{proof}
 By the definition, see \eqref{eq2.15}, \eqref{eq3.12}, $S_A^X\circ I\circ R_{A_\zeta}\circ E_K^X$ is a bounded linear operator from $L^\infty(K,X)$ in $H^\infty(\Di\times \Co\setminus\bar{\Di}_{\frac{r}{6M_\zeta}},X)\cong H^\infty(\Co\setminus\bar{\Di}_{\frac{r}{6M_\zeta}},H^\infty)$.
 We define for $f\in L^\infty(K,X)$, $w\in \Co\setminus\bar{\Di}_{\frac{r}{6M_\zeta}}$, $z\in\Di$,
 \begin{equation}\label{eq3.16}
\bigl((H(w))f\bigr)(z):= \lim_{\varepsilon\rightarrow 0^+}
\frac{1}{2\pi i}\oint_{|\xi|=\frac{r}{6M_\zeta}+\varepsilon} \frac{((S_A^X\circ I\circ R_{A_\zeta}\circ E_K^X)f)(z,\xi)}{w-\xi}\,d\xi.
 \end{equation}
 Then using \eqref{eq2.19} and Lemma \ref{lem3.1} we obtain
 \begin{equation}\label{eq3.17}
 \sup_{w\in  \Co\setminus\bar{\Di}_{\frac{r}{6M_\zeta}}}\left( \sup_{z\in\Di}|(H(w))f\bigr)(z)|\right)\le\frac{12c(\epsilon) M_\zeta}{1-c(\epsilon)^2}\cdot \|f\|_\infty .
 \end{equation}
 Equations \eqref{eq3.16}, \eqref{eq3.17} show that $H\in H^\infty\bigl(\Co\setminus  \bar{\Di}_{\frac{r}{6M_\zeta}},\mathscr B_K^X\bigr)$, vanishes at $\infty$ and has norm $\|H\|_\infty \le \frac{12c(\epsilon) M_\zeta}{1-c(\epsilon)^2}$.
Now, equation \eqref{eq3.15} follows from \eqref{eq2.20} and \eqref{eq3.12}.
 \end{proof}
\sect{Theorems \ref{teo1.1} and \ref{teo1.3} for Sets with Chains of Large Characteristic}
\subsect{} Let $\zeta=\{z_j\}_{j\in\N}\subset K$ be an $\epsilon$-chain of a Lebesgue measurable set $K\subset\Di$ with respect to $\rho$. In this and the next sections we prove Theorems \ref{teo1.1} and \ref{teo1.3} under the assumption 
\begin{equation}\label{eq4.1}
\delta(\zeta)\ge 1-\frac{(1-\sqrt{\epsilon_*})^2}{8};
\end{equation}
here $\epsilon_*:=\max\{\frac 12,\epsilon\}$.

Specifically, in this setting, first we prove 
\begin{Th}\label{teo4.1}
There is a bounded linear operator $L_K^X: L^\infty(K,X)\to C_\rho(\Di,X)$ 
satisfying conditions (i)--(iii) of Theorem \ref{teo1.1} of norm 
\[
\|L_K^X\|\le c\cdot\frac{\epsilon}{1-\epsilon}\quad {\rm for\ some}\quad c<389423
\] 
such that for every $f\in L^\infty(K,X)$ the function $L_K^X f$ is a weak solution of equation \eqref{eq1.1}. 
\end{Th}
We require the following result.
\begin{Lm}\label{lem4.2}
Assume that $\{z_1,\dots, z_k\}$ is an $L$-chain with respect to $\rho$ in the open pseudohyperbolic disk $D(z,R)\subset\Di$, $0<R<1$.
Then
\begin{equation}\label{eq4.2}
\frac{R^2}{1-R^2} \cdot\frac{1-L^2}{L^2}
\le k\le \frac{(2R+L)^2}{1-R^2} \cdot\frac{1}{L^2}.
\end{equation}
\end{Lm}
\begin{proof}
By the definition of an $L$-chain, the pseudohyperbolic disks $D(z_i, \frac{L}{2})$, $1\le i\le k$, are pairwise disjoint so that
\begin{equation}\label{eq4.3}
\bigsqcup_{i=1}^kD(z_i, \mbox{$\frac{L}{2}$})\subset D(z,R'),
\end{equation}
where $R'=\frac{R+\frac L2 }{1+\frac{RL}{2}}$ (for the triangle inequality for $\rho$, see, e.g., \cite[Lm.\,1.4]{Ga}).

Similarly,
\begin{equation}\label{eq4.4}
D(z,R)\subset \bigcup_{i=1}^kD(z_i,L).
\end{equation}

Next, the 2-form $\omega:=\frac{i}{2}\frac{dz\wedge d\bar z}{(1-|z|^2)^2}=\frac{dx\wedge dy}{(1-(x^2+y^2))^2}$ is invariant with respect to the action of the group of M\"{o}bius transformations of $\Di$ and, hence, using the polar coordinates $z=re^{i\phi}\in\Di$ we get for every $s\in (0,1)$,
\[
\int_{D(x,s)}\omega=\int_{\Di_s}\omega=\int_0^{2\pi}\left(\int_0^s \frac{r\, dr}{(1-r^2)^2}\right)\, d\phi=2\pi\cdot \frac 12\cdot\frac{1}{1-r^2}\biggr\rvert_{0}^s=\frac{\pi s^2}{1-s^2}.
\]
Using this formula, we obtain from \eqref{eq4.3}, \eqref{eq4.4} 
\[
k\frac{\pi (\frac L2 )^2}{1-(\frac L2 )^2} \le \frac{\pi (R')^2}{1-(R')^2}=\frac{\pi (2R+L)^2}{(4-L^2)(1-R^2)}\qquad {\rm and}\qquad \frac{\pi R^2}{1-R^2} \le k\frac{\pi L^2}{1-L^2} .
\]
These imply inequalities \eqref{eq4.2}.
\end{proof}
\begin{proof}[Proof of Theorem \ref{teo4.1}]
We choose in Lemma  \ref{lem2.3}
\begin{equation}\label{eq4.5}
\delta:=1-\frac{(1-\sqrt{\epsilon_*})^2}{8},\qquad \lambda:=\sqrt{\epsilon_*}-\frac{1-\epsilon_*}{4}.
\end{equation}
Then
\[
\frac{2\lambda}{1+\lambda^2}=1-\frac{(1-\lambda)^2}{1+\lambda^2}<1-\frac{(1-\sqrt{\epsilon_*})^2(1+\frac{1+\sqrt{\epsilon_*}}{4})^2}{1+\epsilon_*}<1-\frac{25(1-\sqrt{\epsilon_*})^2}{32}<\delta,
\]
as required.

Further, we have
\begin{Lm}\label{lem4.3}
\begin{equation}\label{eq4.6}
r:=\frac{\delta-\lambda}{1-\lambda\delta}\lambda>\epsilon_* .
\end{equation}
\end{Lm}
\begin{proof}
By the definition,
\[
\begin{array}{l}
\displaystyle
1-\lambda\delta=1-\left(1-\frac{(1-\sqrt{\epsilon_*})^2}{8}\right)\cdot\left(\sqrt{\epsilon_*}-\frac{1-\epsilon_*}{4}\right)\\
\\
\displaystyle \le (1-\sqrt{\epsilon_*})+\frac{1-\sqrt{\epsilon_*}}{4}\left(1+\sqrt{\epsilon_*}+\frac{(1-\sqrt{\epsilon_*})\sqrt{\epsilon_*}}{2}\right)=(1-\sqrt{\epsilon_*})\left(1+\frac{2+3\sqrt{\epsilon_*}-\epsilon_*}{8}\right)\\
\\
\displaystyle <\frac{3(1-\sqrt{\epsilon_*})}{2}.
\end{array}
\]
Also,
\[
\begin{array}{l}
\displaystyle
\delta-\lambda=(1-\sqrt{\epsilon_*})+\frac{1-\sqrt{\epsilon_*}}{4}\left(1+\sqrt{\epsilon_*}- \frac{1-\sqrt{\epsilon_*}}{2}\right)=(1-\sqrt{\epsilon_*})\left(\frac{9+3\sqrt{\epsilon_*}}{8}\right).
\end{array}
\]
Hence,
\[
\begin{array}{l}
\displaystyle r-\epsilon_*=\frac{\delta-\lambda}{1-\lambda\delta}\lambda-\epsilon_*>\frac{2}{3(1-\sqrt{\epsilon_*})}\cdot\frac{(9+3\sqrt{\varepsilon_*})(1-\sqrt{\epsilon_*})}{8}\cdot\left(\sqrt{\epsilon_*}-\frac{1-\epsilon_*}{4}\right)-\epsilon_*\\
\\
\displaystyle \qquad\  \ =
\frac{3+\sqrt{\varepsilon_*}}{4}\left(\sqrt{\epsilon_*}-\frac{1-\epsilon_*}{4}\right)-\epsilon_*=\frac{1-\sqrt{\varepsilon_*}}{4}\left(3\sqrt{\epsilon_*}-\frac{(3+\sqrt{\varepsilon_*})(1+\sqrt{\epsilon_*})}{4}\right)\\
\\
\displaystyle \qquad \ \ =\frac{1-\sqrt{\varepsilon_*}}{16}(8\sqrt{\varepsilon_*}-\epsilon_* -3)\ge \frac{1-\sqrt{\varepsilon_*}}{16}\left(8\frac{1}{\sqrt{2}}-\frac 12-3\right)>0
\end{array}
\]
because the function $t\mapsto (-t^2+8t-3)$ is increasing for $t\in[\frac{1}{\sqrt 2},1)$.
\end{proof}
Next, we prove the following result.
\begin{Lm}\label{lem4.4}
\begin{equation}\label{eq4.7}
\delta_m:=\frac{\delta-\frac{2\lambda}{1+\lambda^2}}{1-\frac{2\delta \lambda}{1+\lambda^2}}>\frac{1}{2}.
\end{equation}
\end{Lm}
\begin{proof}
Using the implication 
\[
r=\frac{\delta-\lambda}{1-\lambda\delta}\lambda\quad \Longrightarrow\quad\delta=\frac{\frac{r}{\lambda}+\lambda}{1+r}
\]
we obtain by Lemma \ref{lem4.3}
\[
\delta_m=\frac{\delta-\frac{2\lambda}{1+\lambda^2}}{1-\frac{2\delta \lambda}{1+\lambda^2}}=\frac{\frac{r}{\lambda}-\lambda}{1-r}>\frac{\frac{\epsilon_*}{\sqrt{\epsilon_*}}(1+\frac{1-\epsilon_*}{4\sqrt{\epsilon_*}})-(\sqrt{\epsilon_*}-\frac{1-\epsilon_*}{4})}{1-\epsilon_*}=\frac 12.
\]
\end{proof}
From the previous estimates as a corollary of Lemma \ref{lem2.3} we obtain:
\begin{Lm}\label{lem4.5}
Let $B_\zeta$ be the interpolating Blaschke product with zeros $\zeta=\{z_n\}_{n\in N}$.
Then for  $\lambda$, $\delta$ and $r$ as in \eqref{eq4.5}, \eqref{eq4.6}
\begin{itemize}
\item[(i)]
$B_\zeta^{-1}(\Di_{r})=\{z\in\Co\, :\, |B_\zeta(z)|<r\}$ is the union of pairwise disjoint domains $V_{\zeta n}\,(\ni z_n)$ such that 
\[
B_\zeta^{-1}(\Di_{\epsilon_*})\subset V_{\zeta n}\subset D(z_n,\lambda);\smallskip
\]
\item[(ii)]
$B_\zeta$ maps every $ V_{\zeta n}$ biholomorpically onto $\Di_{r}$;\smallskip
\item[(iii)] 
Every sequence $\omega=\{ w_n\}_{n\in N}$ with $w_n\in D(z_n,\lambda)$ for all $n$,
is  interpolating for $H^\infty$ and
 \[
\delta(\omega)\ge\delta_m>\frac 12.
\]
\end{itemize}
\end{Lm}

We set for $0<\nu\le 2-\sqrt 3$,
\begin{equation}\label{eq4.8}
\epsilon_\nu:=\frac{(2-\sqrt 3)^3\cdot\nu}{6}
\end{equation}
and define a sequence $\zeta_\nu\subset K$ as follows:
\begin{equation}\label{eq4.9}
\zeta_\nu:=\left\{
\begin{array}{ccc}
\zeta&{\rm if}&\epsilon\le\epsilon_\nu\smallskip\\
{\rm an\ \epsilon_\nu-chain\ of}\ K&{\rm if}&\epsilon>\epsilon_\nu.
\end{array}
\right.
\end{equation}
Then, according to Lemma \ref{lem4.2}, for every $z\in\zeta$ the (nonempty) set $\zeta_\nu\cap D(z,\epsilon)$ is finite of cardinality
\begin{equation}\label{eq4.10}
k_z=1\quad {\rm if}\quad \epsilon\le \epsilon_\nu\qquad {\rm and}\qquad
k_z\le \frac{(2\epsilon+\epsilon_\nu)^2}{\epsilon_\nu^2(1-\epsilon^2)}=:k_\nu \quad {\rm if}\quad \epsilon >\epsilon_\nu.
\end{equation}

Let 
\begin{equation}\label{eq4.11}
k_\nu^*:=\max_{z\in\zeta}k_z.
\end{equation}
Equation \eqref{eq4.10} easily implies that $\zeta_\nu$ can be presented as the disjoint union of subsequences $\zeta_\nu^i$, $1\le i \le k_\nu^*\,(\le k_\nu)$, such that for every $z\in\zeta$ the set $\zeta_\nu^i\cap D(z,\epsilon)$
is either empty or consists of a single element.

Given $1\le i\le k_\nu^*$, let $\zeta^i$ be the subset of $\zeta$ such that
 for every $z\in \zeta^i$ the set $\zeta_\nu^i\cap D(z,\epsilon)$ consists of a single element. Since $\delta(\zeta^i)\ge\delta(\zeta)\ge\delta$, Lemma \ref{lem4.5}\,(iii) implies that
 \begin{equation}\label{eq4.12}
 \delta(\zeta_\nu^i)>\frac 12\quad {\rm for\ all}\quad 1\le i\le k_\nu^*.
 \end{equation}

Next, we require
\begin{Lm}\label{lem4.6}
Given $1\le i\le k_\nu^*$ there exists a $\lambda\in (0,\nu)$ in Lemma \ref{lem2.3} such that for the corresponding $r:=r(\lambda)$,
\begin{equation}\label{eq4.13}
\frac{r}{6M_{\zeta_\nu^i}}=\epsilon_\nu .
\end{equation}
\end{Lm}
\begin{proof}
By the definition,
\begin{equation}\label{eq4.14}
r(\lambda):=\frac{\delta(\zeta_\nu^i)-\lambda}{1-\lambda\delta(\zeta_\nu^i)}\cdot\lambda,\quad {\rm where}\quad \frac{2\lambda}{1+\lambda^2}<\delta(\zeta_\nu^i).
\end{equation}
The latter inequality implies that
\[
\lambda<\frac{1}{\delta(\zeta_\nu^i)}-\sqrt{\left(\frac{1}{\delta(\zeta_\nu^i)}\right)^2-1}=:\lambda_{\nu}^i.
\]
Using \eqref{eq4.12} we obtain
\[
\lambda_{\nu}^i>\frac{1}{\frac 12}-\sqrt{\left(\frac{1}{\frac 12}\right)^2-1}=2-\sqrt 3\ge\nu .
\]
This implies that $r$ is a continuous function in $\lambda\in [0,\nu]$.

Also, $r(0)=0$ and
\begin{equation}\label{eq4.15}
r(\nu)=\frac{\delta(\zeta_\nu^i)-\nu}{1-\nu\delta(\zeta_\nu^i)}\cdot\nu>\frac{\frac 12 -(2-\sqrt 3 )}{1-(2-\sqrt 3 )\cdot\frac 12 }\cdot\nu=(2-\sqrt 3)\cdot\nu.
\end{equation}

Moreover, according to \eqref{eq1.5}
\begin{equation}\label{eq4.16}
M_{\zeta_\nu^i}\le \left(\frac{1+\sqrt{1-(\delta(\zeta_\nu^i))^2}}{\delta(\zeta_\nu^i)}\right)^2<\left(\frac{1+\sqrt{1-(\frac 12)^2}}{\frac 12}\right)^2=(2+\sqrt 3 )^2.
\end{equation}

Hence,
\begin{equation}\label{eq4.17}
\frac{r(\nu)}{6M_{\zeta_\nu^i}}>\frac{(2-\sqrt 3)\cdot\nu}{6(2+\sqrt 3 )^2}=\frac{(2-\sqrt 3)^3\cdot\nu}{6}=\epsilon_\nu.  
\end{equation}
Thus by the intermediate value theorem applied to the continuous function $\frac{r(\lambda)}{6M_{\zeta_\nu^i}}$, $\lambda\in [0,\nu]$, we obtain that there is some $\lambda\in (0,\nu)$ such that for the corresponding $r:=r(\lambda)$,
\[
\frac{r}{6M_{\zeta_\nu^i}}=\epsilon_\nu,
\]
as required.
\end{proof}

We apply Theorem \ref{teo3.3} and Corollary \ref{cor3.4} with the $r$ of Lemma \ref{lem4.6}. Then we obtain for 
\begin{equation}\label{eq4.18}
K_\nu^i:=K\cap\bigcup_{z\in\zeta_\nu^i}D(z,\epsilon_\nu)\, (\subset B_{\zeta_\nu^i}^{-1}(\epsilon_\nu))
\end{equation}
(here the implication in the brackets is due to the Schwarz-Pick theorem):\smallskip

\begin{itemize}
\item[(*)] {\em There is  a bounded linear operator $L_{K_\nu^i}^X: L^\infty(K_\nu^i,X)\to C_\rho(\Di,X)$ satisfying conditions (i)--(iii) of Theorem \ref{teo1.1} of norm 
\begin{equation}\label{eq4.19}
\|L_{K_\nu^i}^X\|\le\frac{12\epsilon_\nu M_{\zeta_\nu^i}}{1-\epsilon_\nu^2}
\end{equation}
such that for every $f\in L^\infty(K_\nu^i,X)$ the function $L_{K_\nu^i}^X f$ is a weak solution of equation \eqref{eq1.1}.}\smallskip
\item[(**)] {\em There exists a holomorphic function $h_{\nu}^i\in
 H^\infty\bigl(\Co\setminus  \bar{\Di}_{\epsilon_\nu},\mathscr B_{K_\nu^i}^X\bigr)$ vanishing at $\infty$  of norm}
 \begin{equation}\label{eq4.20}
 \|h_{\nu}^i\|_\infty\le\frac{12\epsilon_\nu M_{\zeta_\nu^i}}{1-\epsilon_\nu^2}
 \end{equation}
 {\em such that}
 \begin{equation}\label{eq4.21}
 (L_{K_\nu^i}^Xf)(z)=\bigl(h_{\nu}^i(B_{\zeta_\nu^i}(z))f\bigr)(z),\quad z\in B_{\zeta_\nu^i}^{-1}\bigl(\Di\setminus  \bar{\Di}_{\epsilon_\nu}\bigr),\quad f\in L^\infty(K_\nu^i,X).
 \end{equation}
 \end{itemize}

Further, we use that

\[
K=\bigcup_{0\le i\le k_\nu^*} K_\nu^i,\quad {\rm where}\quad K_\nu^0:=\emptyset.
\] 
Let $\chi^\nu_j\in L^\infty(\Di)$ be the equivalence class of the characteristic function of the set $S_\nu^j:=K_\nu^j\setminus \cup_{i=0}^{j-1}K_\nu^i$, $1\le j\le k_\nu^*$ (here each $S_\nu^j\ne\emptyset$ by the definition of a pseudohyperbolic chain of $K$). 
Sets $S_\nu^j$ are Lebesgue measurable, mutually disjoint and cover $K$; hence, $\sum_{j=1}^{k_\nu^*}\chi^\nu_j=1$. 
Using the bounded linear projections
\[
M_{\chi^\nu_j}: L^\infty(K,X)\rightarrow L^\infty(K_\nu^j,X),\qquad M_{\chi_\nu^j}f:=\chi_j^\nu\cdot f, \quad 1\le j\le k_\nu^*,
\]
and operators $L_{K_\nu^j}^X: L^\infty(K_\nu^j,X)\rightarrow C_\rho(\Di, X)$, $1\le j\le k_\nu^*$, of statement (*) 
we define
\begin{equation}\label{eq4.22}
L_{K;\nu}^X :=\sum_{j=1}^{k_\nu^*}L_{K_\nu^j}^X\circ M_{\chi^\nu_j}.
\end{equation}
By the definition, $L_{K;\nu}^X: L^\infty(K,X)\rightarrow C_\rho(\Di,X)$ is a bounded linear operator satisfying the statement of Theorem \ref{teo4.1}. To get the required upper bound of the operator norm we define 
\begin{equation}\label{eq4.23}
L_K^X:=L_{K;\tilde\nu}^X,\qquad \tilde\nu:=2-\sqrt 3.
\end{equation}
Then using equations \eqref{eq4.8}--\eqref{eq4.10} and \eqref{eq4.16} we obtain
\[
\begin{array}{l}
\displaystyle 
\|L_K^X\|\le  \sum_{j=1}^{k_{\tilde\nu}^*}\|L_{K_{\tilde\nu}^j}^X\circ M_{\chi^\nu_j}\|\smallskip\\
\qquad\ \le
\left\{
\begin{array}{ccc}
\displaystyle \frac{12\epsilon M_{\zeta}}{1-\epsilon^2}\le
\frac{12\epsilon(2+\sqrt 3 )^2}{1-\epsilon^2}\le\frac{167\epsilon}{1-\epsilon}&{\rm if}&\epsilon\le \epsilon_{\tilde\nu}
\\
\\
\displaystyle
 k_{\tilde\nu}\cdot \frac{2(2-\sqrt 3)^2}{1-\epsilon_{\tilde\nu}^2}\le  \frac{(2+\epsilon_{\tilde\nu})^2 (2-\sqrt 3)^2}{\epsilon_{\tilde\nu}^2 (1-\epsilon_{\tilde\nu}^2)}\cdot\frac{\epsilon}{1-\epsilon}<\frac{389423\epsilon}{1-\epsilon}&{\rm if}&\epsilon>\epsilon_{\tilde\nu}.
 \end{array}
 \right.
 \end{array}
\]
This completes the proof of the theorem.
\end{proof}
\subsect{} In this section, we prove Theorem \ref{teo1.3} under assumption \eqref{eq4.1}. We retain notation of the previous section, see \eqref{eq4.8}, \eqref{eq4.11}, \eqref{eq4.12}, \eqref{eq4.20}.

\begin{Lm}\label{lem4.7}
The following hold:
\begin{equation}\label{eq4.24}
 k_\nu^*\le \frac{c}{\nu^2(1-\epsilon)}\quad {\rm for\ some}\quad c<194712;
  \end{equation}
  \begin{equation}\label{eq4.25}
 K \subset \bigcup_{i=1}^{k_\nu^*}B_{\zeta_\nu^i}^{-1}(\bar\Di_{\epsilon_\nu})\subset\bigcup_{i=1}^{k_\nu^*} B_{\zeta_\nu^i}^{-1}(\Di_{6\epsilon_\nu})\subset [K]_\nu;
 \end{equation}
\begin{equation}\label{eq4.26}
\|h_\nu^i\|_\infty \le \frac{3}{5}\cdot \nu,\qquad 
1\le i\le k_\nu^*.
\end{equation}
\end{Lm}
\begin{proof}
By the definition, see \eqref{eq4.8}, \eqref{eq4.10}, 
\[
k_\nu^*\le  \frac{(2\epsilon+\epsilon_\nu)^2}{\epsilon_\nu^2(1-\epsilon^2)}= \frac{\bigl(6(2+\sqrt 3)^3\bigr)^2(2\epsilon+\epsilon_\nu)^2}{\nu^2 (1-\epsilon)(1+\epsilon)}\le \frac{\bigl(6(2+\sqrt 3)^3\bigr)^2(2+\frac{(2-\sqrt 3)^4}{6})^2}{\nu^2 (1-\epsilon)\cdot 2}\le
\frac{194712}{\nu^2 (1-\epsilon)},
\]
as stated.

Further, according to the Schwarz-Pick theorem and Lemma \ref{lem2.3}\,(i) with $\lambda$ and $r$ as in Lemma \ref{lem4.6}, given $1\le i\le k_\nu^*$ we have
\[
K_\nu^i\subset [\zeta_\nu^i]_{\epsilon_\nu}\subset B_{\zeta_\nu^i}^{-1}(\bar\Di_{\epsilon_\nu})\subset B_{\zeta_\nu^i}^{-1}(\Di_{6\epsilon_\nu})
\subset [\zeta_\nu^i]_\nu\subset [K]_\nu.
\]
This and \eqref{eq4.18} imply \eqref{eq4.25}:
\[
K:=\bigcup_{i=1}^{k_\nu^*}K_\nu^i
\subset \bigcup_{i=1}^{k_\nu^*}B_{\zeta_\nu^i}^{-1}(\bar\Di_{\epsilon_\nu})\subset \bigcup_{i=1}^{k_\nu^*}B_{\zeta_\nu^i}^{-1}(\Di_{6\epsilon_\nu})\subset [K]_\nu .
\]

Finally, due to \eqref{eq4.8}, \eqref{eq4.16}, \eqref{eq4.20},
\[
\|h_\nu^i\|_\infty\le\frac{12\epsilon_\nu M_{\zeta_\nu^i}}{1-\epsilon_\nu^2}\le\frac{2\cdot(2-\sqrt 3)\cdot\nu}{1-\frac{(2-\sqrt 3)^8}{36}}\le\frac{3}{5}\cdot\nu.
\]
The proof of the lemma is complete.
\end{proof}
Let us define holomorphic functions $ H_\nu^i\in H^\infty\bigl(\Co\setminus  \bar{\Di}_{\epsilon_\nu},\mathscr B_{K}^X\bigr)$ vanishing at $\infty$ by the formulas
\begin{equation}\label{eq4.27}
H_\nu^i(w):=h_{\nu}^i(w)\circ M_{\chi^\nu_j},\qquad w\in \Co\setminus  \bar{\Di}_{\epsilon_\nu},\qquad 1\le i\le k_\nu^*.
\end{equation}
Then \eqref{eq4.26} implies that
\begin{equation}\label{eq4.28}
\| H_\nu^i\|_\infty \le \frac{3}{5}\cdot \nu,\qquad 1\le i\le k_\nu^*.
\end{equation}

Now, the required version of Theorem \ref{teo1.3} reads as follows:
\begin{Th}\label{teo4.8}
Under assumption \eqref{eq4.1} and in notation of Lemma \ref{lem4.7}
 there exists  an operator $E_{\nu}^0\in \mathscr B(K,X)$ such that
 for every $f\in L^\infty(K,X)$, $z\in \Di\setminus\overline{[K]}_\nu$
\begin{equation}\label{eq4.29}
 (L_K^Xf)(z)=(E_{\nu}^0 f)(z)+\sum_{i=1}^{k_\nu^*}\bigl(H_\nu^i(B_{\zeta_\nu^i}(z))f\bigr)(z).
 \end{equation}
 Here $k_\nu^*$, $B_{\zeta_\nu^i}$ and $ H_\nu^i$  satisfy \eqref{eq4.24}, \eqref{eq4.25}, \eqref{eq4.27}, \eqref{eq4.28}.
 \end{Th}
\begin{proof}[Proof of Theorem \ref{teo4.8}]
We define, see \eqref{eq4.22}, \eqref{eq4.23},
\begin{equation}\label{eq4.30}
E_{\nu}^0:=L_K^X-L_{K;\nu}^X.
\end{equation}
Then $E_{\nu}^0$ is a bounded linear operator from $L^\infty(K,X)$ in $H^\infty(\Di,X)$.

Next, due to \eqref{eq4.25},
\begin{equation}\label{eq4.31}
\Di\setminus \overline{[K]}_\nu\subset \bigcap_{i=1}^{k_\nu^*}B_{\zeta_\nu^i}^{-1}(\Di\setminus\Di_{6\epsilon_\nu})\subset \bigcap_{i=1}^{k_\nu^*} B_{\zeta_\nu^i}^{-1}(\Di\setminus\bar{\Di}_{\epsilon_\nu}).
\end{equation}
Hence, all functions $h_\nu^i\circ B_{\zeta_\nu^i}$ are defined on $\Di\setminus \overline{[K]}_\nu$. Therefore
statement (**) and formulas \eqref{eq4.21}, \eqref{eq4.22} imply that
\[
(L_K^Xf)(z)=(E_{\nu}^0 f)(z)+\sum_{i=1}^{k_\nu^*}\bigl(\bigl(h_\nu^i(B_{\zeta_\nu^i}(z))\circ M_{\chi^\nu_j}\bigr)f\bigr)(z),\quad z\in \Di\setminus\overline{[K]}_\nu, \quad f\in L^\infty(K,X).
\]
This completes the proof of the theorem.
\end{proof}

\sect{Proofs of Theorems \ref{teo1.1} and \ref{teo1.3}}
In this section, we prove Theorems \ref{teo1.1} and \ref{teo1.3} in the general case.
\begin{proof}[Proof of Theorem \ref{teo1.1}]
We use the following result.
\begin{Lm}[\mbox{\cite[Ch.\,X,\,Cor.\,1.6]{Ga}}]\label{lem5.1}
Every interpolating sequence $\zeta$ containing at least two elements can be presented as the disjoint union of interpolating sequences $\zeta_1$ and $\zeta_2$ such that
\[
\delta(\zeta_j)\ge\sqrt{\delta(\zeta)},\quad j=1,2.
\]
\end{Lm}

For $\delta$ and $\epsilon$ as in Theorem \ref{teo1.1},
let $l\in\N$ be such that 
\begin{equation}\label{eq5.1}
\frac{\log \delta}{\log\bigl(1-\frac{(1-\sqrt{\epsilon_*})^2}{8}\bigr)}\le 2^l< \frac{2\log\delta}{\log\bigl(1-\frac{(1-\sqrt{\epsilon_*})^2}{8}\bigr)}
\end{equation}
if $\delta\le 1-\frac{(1-\sqrt{\epsilon_*})^2}{8}$ and $l=0$ otherwise.

Then
\[
\delta_l:=\delta^{\frac{1}{2^l}}\ge 1-\frac{(1-\sqrt{\epsilon_*})^2}{8}.
\]
If $\zeta\subset K$, $\delta(\zeta)\ge\delta$, is an $\epsilon$-chain of $K$ with respect to $\rho$  of Theorem \ref{teo1.1}, then due to Lemma \ref{lem5.1} we can present $\zeta$ as the disjoint union of $s\le 2^l$
subsets $\zeta_j$, $1\le j\le s$, such that $\delta(\zeta_j)\ge\delta_l$ for all $j$. Let
\begin{equation}\label{eq5.2}
K_j:=K\cap \left(\bigcup_{z\in\zeta_j}D(z,\epsilon)\right),\quad 1\le j\le s.
\end{equation}
Then $\zeta_j$ is an $\epsilon$-chain of $K_j$ and
\[
K=\bigcup_{j=0}^{s} K_j,\qquad K_0:=\emptyset.
\]

Let $\chi_j\in L^\infty(\Di)$ be the equivalence class of the characteristic function of the set $R_j:=K_j\setminus \cup_{i=1}^{j-1}K_i$, $1\le j\le s$ (here each $R_j\ne\emptyset$ by the definition of an $\epsilon$-chain).
Since sets $R_j$ are Lebesgue measurable, mutually disjoint and cover $K$, 
\begin{equation}\label{eq5.3}
\sum_{j=1}^s\chi_j=1.
\end{equation}
Applying Theorem \ref{teo4.1} to $K_j$
we define
\begin{equation}\label{eq5.4}
L_K^X:=\sum_{j=1}^{s}L_{K_j}^X\circ M_{\chi_j},
\end{equation}
where
\[
M_{\chi_j}: L^\infty(K,X)\rightarrow L^\infty(K_j,X),\qquad M_{\chi_j}f:=\chi_j\cdot f, \quad 1\le j\le s,
\]
is a bounded linear projection.

By the definition, $L_K^X: L^\infty(K,X)\rightarrow C_\rho(\Di,X)$ is a linear operator satisfying the statement of Theorem \ref{teo1.1}
of norm, see Theorem \ref{teo4.1} and \eqref{eq5.1},
\[
\begin{array}{l}
\displaystyle
\|L_K^X\|\le 
\sum_{j=1}^{s}\|L_{K_j}^X\|\le \frac{389423\epsilon}{1-\epsilon}\cdot
\max\left\{1,\frac{2\log \frac 1\delta}{\log\frac{1}{1-\frac{(1-\sqrt{\epsilon_*})^2}{8}}}\right\}\le \frac{389423\epsilon}{1-\epsilon}\cdot
\max\left\{1,\frac{64\log \frac 1\delta}{(1-\epsilon_*)^2}\right\}\\
\\
\displaystyle \le \frac{24923072\epsilon}{1-\epsilon}\cdot
\max\left\{1,\frac{\log \frac 1\delta}{(1-\epsilon_*)^2}\right\}.
\end{array}
\]
We used here the inequalities $\log(1+t)\le t$ for $t\ge 0$ and $\max\{1,ab\}\le a\cdot\max\{1,b\}$ for $a\ge 1$, $b\ge 0$.

The proof of the theorem is complete.
\end{proof}

\begin{proof}[Proof of Theorem \ref{teo1.3}]
The proof goes in the same way as that of Theorem \ref{teo4.8}.
Specifically, if $H_{\nu;j}^{i}$, $B_{\zeta_{\nu;j}^i}$, $1\le i\le k_{\nu;j}^*$, are the corresponding objects of Theorem \ref{teo4.8} for $K_j$, $1\le j\le s$, then the holomorphic functions of Theorem \ref{teo1.3} are defined as follows, see \eqref{eq5.4},
\begin{equation}\label{eq5.5}
H_\nu^{ij}(w):=H_{\nu;j}^{i}(w)\circ M_{\chi_j}\qquad w\in \Co\setminus  \bar{\Di}_{\epsilon_\nu},\qquad 1\le i\le k_{\nu;j}^*,\quad 1\le j\le s.
\end{equation}
By the definition, see \eqref{eq4.24}, \eqref{eq5.1}, the number $k_\nu^*$ of such functions  is
\begin{equation}
\sum_{j=1}^s k_{\nu;j}^*\le \frac{194712}{\nu^2(1-\epsilon)}\cdot\max\left\{1,\frac{2\log \frac 1\delta}{\log\frac{1}{1-\frac{(1-\sqrt{\epsilon_*})^2}{8}}}\right\}\le \frac{12461568}{\nu^2(1-\epsilon)}\cdot\max\left\{1,\frac{\log\frac{1}{\delta}}{(1-\epsilon_*)^2}\right\}.
\end{equation}

Next, if $E_{\nu;j}^0\in\mathscr B_{K_j}^X$ are operators \eqref{eq4.30} for $K=K_j$, $1\le j\le s$, we define
\begin{equation}\label{eq5.7}
E_\nu^0:=\sum_{j=1}^s E_{\nu;j}^0\circ M_{\chi_j}\in \mathscr B_K^X.
\end{equation}
This, \eqref{eq5.4} and Theorem \ref{teo4.8} imply  \eqref{eq1.12}, i.e.,
for every $f\in L^\infty(K,X)$, $z\in \Di\setminus\overline{[K]}_\nu$,
\begin{equation}\label{eq5.8}
 (L_K^Xf)(z)=(E_{\nu}^0 f)(z)+\sum_{j=1}^s\sum_{i=1}^{k_{\nu;j}^*}\bigl(H_{\nu}^{ij}(B_{\zeta_{\nu;j}^i}(z))f\bigr)(z).
 \end{equation}
 The proof of Theorem \ref{teo1.3} is complete.
\end{proof}

\sect{Proof of Corollary \ref{cor1.5}} 
Let $\mathfrak M$ be the maximal ideal space of $H^\infty$ equipped with the Gelfand topology, the induced weak$^*$ topology on $\mathfrak M$ (see Remark \ref{rem1.6}). Then $\mathfrak M=\mathfrak M_a\sqcup\mathfrak M_s$, where $\mathfrak M_a$ and $\mathfrak M_s$ are sets of nontrivial (maximal analytic disks) and one-pointed Gleason parts for $H^\infty$.  The set $\mathfrak M_a$ is open and dense in $\mathfrak M$ and is the union of closures in $\mathfrak M$ of all interpolating sequences for $H^\infty$, see, e.g., \cite[Ch.\,X]{Ga}. Using the definition of a Gleason part one easily shows that {\em a subset of $\Di$ is quasi-interpolating if and only if its closure in $\mathfrak M$ lies in $\mathfrak M_a$}. 

In what follows, we naturally identify $\Di$ with an open dense subset of $\mathfrak M_a$ formed by evaluation homomorphisms at points of $\Di$.

Let $K$ be a Lebesgue measurable quasi-interpolating set. Given $\nu\in (0,1)$ consider the $\nu$-hyperbolic neighbourhood $[K]_\nu$ of $K$, see \eqref {eq1.7}. Clearly, $[K]_\nu$ is a quasi-interpolating subset of $\Di$ as well; thus, the closure ${\rm cl}([K]_\nu)$ of $[K]_\nu$ in $\mathfrak M$ is a compact subset of $\mathfrak M_a$. 

Let $g\in {\rm range}(L_K)\subset C_\rho(\Di)$, see Theorem \ref{teo1.1}. Then due to \cite[Thm.\,1.2\,(c)]{Br1}  (see also \cite[Thm.\,2.1]{S}), $g$ admits a continuous extension $G$ to $\mathfrak M_a$. 

Next, due to Theorem \ref{teo1.3} (see \eqref{eq1.13}), the restriction 
$g':=g|_{U}$, $U:=\Di\setminus\overline{[K]}_\nu$, is the limit of a uniformly convergent on $U$ sequence of meromorphic functions $\{\frac{g_n}{h_n}\}_{n\in\N}$ such that all $g_n, h_n\in H^\infty$ and all $\frac{1}{|h_n|}$ are bounded from above on $U$. In particular, functions $\frac{g_n}{h_n}$ are continuously extended  via the Gelfand transform $\,\hat{\,}$ (see Remark \ref{rem1.6}) to the open set $\mathfrak M\setminus {\rm cl}([K]_\nu)\subset\mathfrak M$.
By the Carleson corona theorem, $U$ is an open dense subset of $\mathfrak M\setminus {\rm cl}([K]_\nu)$; hence, the extended sequence $\{\frac{\hat g_n}{\hat h_n}\}_{n\in\N}$ converges uniformly on $\mathfrak M\setminus {\rm cl}([K]_\nu)$ to a continuous function $G'$ which extends $g'$. 
Since ${\rm cl}([K]_\nu)\subset\mathfrak M_a$, the union of domains of $G$ and $G'$ is
\[
\mathfrak M_a\cup\bigl(\mathfrak M\setminus {\rm cl}([K]_\nu)\bigr)=\mathfrak M.
\]
In turn, by the Carleson corona theorem, $U$ is dense in the open set $\mathfrak M_a\setminus {\rm cl}([K]_\nu)$, the intersection of domains of $G$ and $G'$.  Hence, $G$ and $G'$ coincide on $\mathfrak M_a\setminus {\rm cl}([K]_\nu)$ (as they are continuous extensions of the same function) and the formula
\[
\widetilde G:=\left\{
\begin{array}{ccc}
G&{\rm on}&\mathfrak M_a\smallskip\\
G'&{\rm on}&\mathfrak M\setminus {\rm cl}([K]_\nu)
\end{array}
\right.
\]
determines a continuous function on $\mathfrak M$ which extends $g$. 

Applying the Stone-Weierstrass theorem, we obtain that $C(\mathfrak M)$ coincides with the algebra $\hat{\mathcal A}:=\{\hat f\in C(\mathfrak M)\, :\, f\in\mathcal A\}$. This implies that $g=\widetilde G|_\Di\in\mathcal A$,
as stated.

This completes the proof of the corollary.

\sect{Banach-valued Corona Problem for $H^\infty(\Di,A)$}
In this section we describe an application of the  nonquantitative version of Theorem \ref{teo1.1} to the Banach-valued corona problem for $H^\infty$ presented in \cite{Br3}. 

Let $A$ be a uniform algebra defined on the maximal ideal space $\mathfrak M(A)$ and $H^\infty(\Di, A)$ be the Banach algebra of bounded $A$-valued holomorphic functions on $\Di$. There is a
continuous embedding $\iota$ of $\Di\times \mathfrak M(A)$ into the maximal ideal space $\mathfrak M(H^\infty(\Di, A))$ taking $(z,x)\in \Di\times\mathfrak M(A)$ to the 
evaluation homomorphism $f\mapsto (f(z))(x)$, $f\in H^\infty(\Di, A)$.  The complement of
the closure of $\iota(X)$ in $\mathfrak M(A)$ is called the {\em corona}.
The {\em corona problem} asks whether the corona is empty.  
The problem can be equivalently reformulated as follows, see, e.g., \cite[Ch.\,V,\,Thm.\,1.8]{Ga}:

A collection $f_{1},\dots, f_{n}\in H^\infty(\Di,A)$ satisfies
the {\em corona condition} if
\begin{equation}\label{eq7.1}
1\ge \max_{1\leq j\leq n}|(f_{j}(z))(x)|\geq\delta>0\ \ \ {\rm for\ all}\ \ \
(z,x)\in\Di\times \mathfrak M(A).
\end{equation}
The corona problem being solvable (i.e., the corona is empty)
means that for all $n\in\mathbb N$ and $f_{1},\dots, f_{n}$ satisfying the corona condition, the Bezout equation
\begin{equation}\label{eq7.2}
f_{1}g_{1}+\cdots+f_{n}g_{n}= 1
\end{equation}
has a solution $g_{1},\dots, g_{n}\in H^\infty(\Di,A)$. 

The next result established in \cite{Br3} asserts that the Bezout equation is solvable if and only if it is locally solvable.
\begin{Th}[\mbox{\cite[Thm.\,6.1]{Br3}}]\label{teo6.1}
Suppose $f_{1},\dots, f_{n}\in H^\infty(\Di,A)$ satisfy \eqref{eq7.1}.
Equation \eqref{eq7.2} is solvable if and only if there exist a finite open cover $(U_j)_{1\le j\le m}$ of $\mathfrak M$ and holomorphic functions
$g_{ij}\in H^\infty(U_j\cap\Di,A)$, $1\le i\le n$, $1\le j\le m$, such that 
\begin{equation}\label{eq7.3}
f_{1}|_{U_j\cap\Di}\cdot g_{1j}+\cdots+f_{n}|_{U_j\cap\Di}\cdot g_{nj}= 1\quad {\rm for\ all}\quad j.
\end{equation}
\end{Th}
The proof of the theorem follows the lines of the proof of \cite[Thm.\,1.11]{Br2}, where instead of \cite[Thm.\,3.5]{Br2} one uses Theorem \ref{teo1.1}.  

In a forthcoming paper we apply Theorem \ref{teo6.1} to the corona problem for the algebra of bounded holomorphic functions on a polydisk.

\medskip

\noindent {\em Data sharing not applicable to this article as no datasets were generated or analysed during the current study.}

\end{document}